\documentclass[12pt]{amsart}

\usepackage[utf8]{inputenc} 
\usepackage[OT2,T1]{fontenc}
\usepackage{ae} 
\usepackage{aecompl}

\usepackage[top=2.5cm,bottom=2.5cm,right=3cm,left=2cm]{geometry}

\usepackage{amsmath,amsthm,amssymb,mathrsfs,calligra,dsfont,mathtools}
\usepackage{fouriernc}
\usepackage[final, babel=true]{microtype}
\usepackage{caption}
\usepackage{subcaption}
\usepackage{eucal}
\usepackage{verbatim}
\usepackage{centernot}
\usepackage{enumitem}
\usepackage{graphicx}
\usepackage{tikz}
\usetikzlibrary{matrix,arrows,cd}
\usepackage[all]{xy}
\usepackage{hyperref}
\usepackage{xcolor}
\usepackage{url}
\definecolor{darkgreen}{rgb}{0,0.45,0}  

\hypersetup{
	colorlinks,
	linktoc=all,
	linkcolor={darkgreen},
	citecolor={blue!50!black},
	urlcolor={blue!80!black}
} 

\DeclareFontFamily{U}{mathx}{\hyphenchar\font45}
\DeclareFontShape{U}{mathx}{m}{n}{
	<5> <6> <7> <8> <9> <10>
	<10.95> <12> <14.4> <17.28> <20.74> <24.88>
	mathx10
}{}
\DeclareSymbolFont{mathx}{U}{mathx}{m}{n}
\DeclareFontSubstitution{U}{mathx}{m}{n}
\DeclareMathAccent{\widecheck}{0}{mathx}{"71}
\DeclareMathAccent{\wideparen}{0}{mathx}{"75}

\newcommand{\Q}{\mathbb{Q}}

\newcommand{\Z}{\mathbb{Z}}

\newcommand\bR{\mathbb{R}}
\newcommand\bZ{\mathbb{Z}}
\newcommand\bQ{\mathbb{Q}}
\newcommand\bP{\mathbb{P}}

\newcommand\Zp{{\mathbb{Z}_p}}
\newcommand\Qp{{\mathbb{Q}_p}}

\newcommand\cL{\mathcal{L}}
\newcommand\cG{\mathcal{G}}
\newcommand\cH{\mathcal{H}}
\newcommand\cN{\mathcal{N}}
\newcommand\cO{\mathcal{O}}
\newcommand\cT{\mathcal{T}}

\newcommand\gp{\mathfrak{p}}

\newcommand\gP{\mathfrak{P}}

\newcommand\sC{\mathscr{C}}
\newcommand\sL{\mathscr{L}}
\newcommand\sZ{\mathscr{Z}}

\newcommand{\ob}[1]{\mkern 1.5mu\overline{\mkern-1.5mu#1\mkern-1.5mu}\mkern 1.5mu}

\newcommand{\loc}{\mathrm{loc}}
\newcommand{\Loc}{\mathrm{Loc}}

\newcommand{\Ind}{\mathrm{Ind}}

\newcommand{\res}{\mathrm{res}}

\newcommand{\ord}{\mathrm{ord}}

\newcommand{\Art}{\textrm{Art}}
\newcommand{\ab}{\textrm{ab}}

\DeclareMathOperator{\Gal}{Gal}
\DeclareMathOperator{\id}{id}
\DeclareMathOperator{\Hom}{Hom}

\DeclareMathOperator{\coker}{coker}
\DeclareMathOperator{\rec}{rec}
\DeclareMathOperator{\rk}{rk}

\DeclareMathOperator{\Tr}{Tr}
\DeclareMathOperator{\GL}{GL}
\DeclareMathOperator{\Diag}{Diag}
\newcommand{\cyc}{\textrm{cyc}}
\newcommand{\HH}{\mathrm{H}}
\newcommand{\ff}{\mathrm{f}}

\newcommand{\ur}{\mathrm{ur}}

\DeclareSymbolFont{cyrletters}{OT2}{wncyr}{m}{n}
\DeclareMathSymbol{\Sha}{\mathalpha}{cyrletters}{"58}

\newtheoremstyle{thmstyle}
{\parskip} % Space above
{\topsep} % Space below
{\itshape} % Body font
{} % Indent amount
{\bfseries} % Theorem head font
{.} % Punctuation after theorem head
{.5em} % Space after theorem head
{} % Theorem head spec (can be left empty, meaning `normal')

\newtheoremstyle{defstyle}
{\parskip} % Space above
{\topsep} % Space below
{} % Body font
{} % Indent amount
{\bfseries} % Theorem head font
{.} % Punctuation after theorem head
{.5em} % Space after theorem head
{} % Theorem head spec (can be left empty, meaning `normal')

\theoremstyle{defstyle}

\newtheorem{exemple}[subsubsection]{Example}
\newtheorem{remark}[subsubsection]{Remark}

\theoremstyle{thmstyle}
\newtheorem{theorem}[subsubsection]{Theorem}
\newtheorem{proposition}[subsubsection]{Proposition}
\newtheorem{lemme}[subsubsection]{Lemma}
\newtheorem{corollaire}[subsubsection]{Corollary}
\newtheorem*{corollaire*}{Corollary}

\newtheorem*{conjecture*}{Conjecture}

\newtheorem{THM}[subsection]{Theorem}
\newtheorem{CONJ}[subsection]{Conjecture}
\newtheorem{CORO}[subsection]{Corollary}

\author{Alexandre Maksoud}
\address{Universität Paderborn, J2.305, Warburger Str. 100, 33098 Paderborn, Germany}
\email{maksoud.alexandre@gmail.com}
\subjclass{11R27 (primary), 11R23 (secondary).}
\keywords{Leopoldt conjecture, Gross-Kuz'min conjecture, $p$-adic transcendence theory, Iwasawa theory, Artin formalism}

\title{On the rank of Leopoldt's and Gross's regulator maps}
\begin{document}
	\begin{abstract}
		We generalize Waldschmidt's bound for Leopoldt's defect and prove a similar bound for Gross's defect for an arbitrary extension of number fields. As an application, we prove new cases of Gross's finiteness conjecture (also known as the Gross-Kuz'min conjecture) beyond the classical abelian case, and we show that Gross's $p$-adic regulator has at least half of the conjectured rank. We also describe and compute non-cyclotomic analogues of Gross's defect.
	\end{abstract}
	\maketitle
	
	\section{Introduction}
	
	Let $p$ be a prime number. Given a number field $K$, we denote by $S_p(K)$ and $S_\infty(K)$ the sets of $p$-adic places and archimedean places of $K$ respectively. Fix an algebraic closure $\ob{\Q}_p$ of $\Qp$ and let $A^\wedge= \ob{\Q}_p \otimes_{\Zp} \varprojlim_n A/p^n A$ for any abelian group $A$. 
	The \emph{Leopoldt regulator map} is the $\ob{\Q}_p$-linear map
	\begin{equation}\label{eq:leopoldts_map}
		\iota_K \colon \cO_K^{\times,\wedge} \longrightarrow \prod_{\gP|p} \cO_{K_\gP}^{\times,\wedge}
	\end{equation}
	induced by the diagonal embedding of the unit group of $K$ into all its $p$-adic completions.
	
	\begin{CONJ}[Leopoldt's conjecture for $(K,p)$, \cite{leopoldt}]
		The Leopoldt regulator map $\iota_K$ is injective.
	\end{CONJ}

	Ax's method, combined with Brumer's $p$-adic analogue of Baker's theorem, implies Leopoldt's conjecture for abelian extensions of $\bQ$ or of an imaginary quadratic field \cite{ax,brumer}. The same method also proves Leopoldt's conjecture when $K$ is an imaginary $A_4$-extension of $\bQ$ \cite{emsalem1984independance}. 
	
	Another classical result concerning Leopoldt's conjecture is obtained by Waldschmidt as an application of his study of transcendence properties of the exponential function in several variables. Let $\delta^\textbf{L}_K$ be the dimension of $\ker \iota_K$. Leopoldt's conjecture then predicts that $\delta^\textbf{L}_K=0$ and $\delta^\textbf{L}_K$ is called Leopoldt's defect (we will omit its dependence on $p$).

	\begin{THM}[Waldschmidt, \cite{waldschmidt}]\label{thm:waldschmidt_bound}
		For every number field $K$, we have
		\[\delta^\textbf{L}_K\leq (|S_\infty(K)|-1)/2.\]
	\end{THM}

	This bound is the best upper bound for Leopoldt's defect so far. 
	The main goal in this work is to derive from Waldschmidt's and Roy's contributions in $p$-adic transcendence theory similar results concerning the Gross regulator map and to prove new cases of the Gross-Kuz'min conjecture, whose statement is now recalled.
	
	Consider the $\Zp$-hyperplane $\cH $ of $\prod_{\gP|p} \Zp$ given by the equation $\sum_{\gP|p}s_\gP =0$, and the map
	\begin{equation}\label{eq:def_sL_gross}
		\sL_K \colon \left\{\begin{array}{ccl}
			\cO_K[\tfrac{1}{p}]^{\times,\wedge} & \longrightarrow & \cH^\wedge \\
			x & \mapsto & (-\log_p(\cN_{\gP}(x)))_\gP,
		\end{array}\right.
	\end{equation}
	where $\cO_K[\tfrac{1}{p}]^\times$ is the group of $p$-units of $K$, $\cN_{\gP}$ is the local norm map for the extension $K_\gP/\Qp$, and $\log_p:\Q_p^\times \to \Qp$ is the usual Iwasawa $p$-adic logarithm. By the usual
	product formula, $\sL_K$ is well-defined. We shall refer to $\sL_K$ as the (cyclotomic) Gross regulator map.
	
	\begin{CONJ}[Gross-Kuz'min's conjecture for $(K,p)$, \cite{gross1981padic,kuzmin}]
		The Gross regulator map $\sL_K$ is surjective.
	\end{CONJ}

	Like Leopoldt's conjecture, the Gross-Kuz'min conjecture plays a central role in the formulation of $p$-adic analogues of Dirichlet's class number formula. Leopoldt's regulator appears in Colmez's formula on the residue at $s=1$ of the $p$-adic Dedekind zeta function of a totally real number field \cite{colmezleopoldt}, whereas Gross's regulator plays the role of an $\cL$-invariant in the celebrated Gross-Stark conjecture over a CM number field \cite{gross1981padic}. When $K$ is neither totally real nor CM, a conjectural interpretation of these regulators in terms of $p$-adic Artin $L$-functions is still available (see \cite{hIMC}). Besides their potential applications to the equivariant Tamagawa number conjecture as in \cite{BKSANT}, the Gross-Kuz'min conjecture and its non-cyclotomic analogue discussed below also yield information on the fine structure of class groups attached to $\Zp$-extensions of $K$ (see e.g. \cite{federergross,kolster,nguyenquangdocapitulation,jaulent2017}).
	
	The Gross-Kuz’min conjecture is true when $K/\Q$ is abelian as shown by Greenberg \cite{greenberg1973}. More recently, Kleine \cite{kleine} proved the conjecture for any $K$ which has at most two $p$-adic primes. (We note that Kleine’s approach does not use $p$-adic transcendence theory).

	Our first result gives an upper bound for the Gross defect $\delta^\textbf{G}_K=\dim \coker \sL_K$ as well as a slight generalization of Theorem \ref{thm:waldschmidt_bound}.

	\begin{THM}\label{THM:generalization_Waldschmidt}
		Let $K/k$ be an extension of number fields. The following inequalities hold:
		\[\delta^\textbf{L}_K \leq \delta^\textbf{L}_k + \left(|S_\infty(K)|-|S_\infty(k)|\right)/2, \quad\quad
		\delta^\textbf{G}_K \leq \delta^\textbf{G}_k + \left(|S_p(K)|-|S_p(k)|\right)/2.\]
%		\[\begin{array}{ccc}
%			\delta^\textbf{L}_K &\leq& \delta^\textbf{L}_k + \left(|S_\infty(K)|-|S_\infty(k)|\right)/2, \\
%			\delta^\textbf{G}_K &\leq& \delta^\textbf{G}_k + \left(|S_p(K)|-|S_p(k)|\right)/2.	
%		\end{array}\]
		Moreover, if $K$ has at least one real place and $|S_p(K)|\neq|S_p(k)|$, then the second bound is strict.
	\end{THM}
	It follows from the last statement $k=\bQ$ that the Gross-Kuz'min conjecture holds for all cubic number fields.
	
	A key observation that we use in the computation of $\delta^\textbf{L}_K$ and $\delta^\textbf{G}_K$ is that they are compatible with Artin formalism. For any number field $k\subset \ob{\Q}$ of absolute Galois group $G_k=\Gal(\ob{\Q}/k)$, let $\Art_{\ob{\Q}_p}(G_k)$ be the set of finite dimensional $\ob{\Q}_p$-valued representations of $G_k$ of finite image. We will define defects $\delta^\textbf{L}_k(\rho)$ and $\delta^\textbf{G}_k(\rho)$ associated with $\rho\in \Art_{\ob{\Q}_p}(G_k)$ which satisfy the usual Artin formalism. In particular, when $\rho=\Ind^k_K \mathds{1}_K$ is the induction from $G_K$ to $G_k$ of the trivial representation, they coincide with $\delta^\textbf{L}_K$ and $\delta^\textbf{G}_K$ respectively. We will also define quantities $d(\rho)$, $d^+(\rho)$ and $f(\rho)$ which compute $[K:\Q]$, $|S_\infty(K)|$ and $|S_p(K)|$, respectively, when $\rho=\Ind^k_K \mathds{1}_K$ (see (\ref{eq:def_invariants_d+_et_f})). Our main theorem is the following.
	
	\begin{THM}\label{THM:main_thm}
		Let $\rho \in \Art_{\ob{\Q}_p}(G_\Q)$ be an irreducible representation and let $d=d(\rho)$, $d^+=d^+(\rho)$ and $f=f(\rho)$. If $d^+=f=0$, then we have $\delta^\textbf{L}_\Q(\rho)=\delta^\textbf{G}_\Q(\rho)=0$. Otherwise, we have the following inequalities.
		\[
		\delta^\textbf{L}_\Q(\rho) \leq \dfrac{(d^+)^2}{d+d^+}, \qquad
		\delta^\textbf{G}_\Q(\rho) \leq \dfrac{f^2}{d^++2f}.	
		\]
	\end{THM}
	By Artin formalism, this yields the upper bound (e.g. for Leopoldt's defect) $\delta^\textbf{L}_k(\rho)\leq d^+(\rho)/2$ for an arbitrary representation $\rho\in \Art_{\ob{\Q}_p}(G_k)$.
	We immediately recover  Theorem \ref{THM:generalization_Waldschmidt} by choosing $\rho$ such that $\Ind^k_K \mathds{1}_K = \rho \oplus \mathds{1}_k$.
	
	The first bound in Theorem \ref{THM:main_thm} is Laurent's main theorem in \cite{laurentcrelle}, but we will provide a much shorter proof of this result via a lemma on local Galois representations (Lemma \ref{lem:fontaine}). The second bound, however, does not seem to follow from the classical methods employed by Laurent \cite{laurentcrelle} and Roy \cite{roy} to study the $p$-adic closure of $S$-units of $K$, for a given finite set of places $S$. 
	
	Theorem \ref{THM:main_thm} together with Artin's formalism easily implies Leopoldt's conjecture for abelian extensions of an imaginary quadratic field. In the same vein, we indicate the two main applications of Theorem \ref{THM:main_thm}. 
	
	\begin{CORO}\label{CORO:gross_reg}
		Let $k$ be a totally real field and let $V$ be a totally odd Artin representation of $G_k$. Then Gross's $p$-adic regulator matrix $R_p(V)$ defined in \cite[(2.10)]{gross1981padic} has rank at least half of its size.
	\end{CORO}
	This corollary strengthens Gross's classical result stating that the matrix $R_p(V)$ has positive rank \cite[Prop. 2.13]{gross1981padic}.
	\begin{CORO}\label{CORO:gk_dihedral}
	The Gross-Kuz'min conjecture holds for abelian extensions of imaginary quadratic fields. It also holds for abelian extensions of real quadratic fields having at least one real place.
	\end{CORO}
	Theorem \ref{thm:list_new_cases_gk} provides a more extensive list of number fields for which the Gross-Kuz'min conjecture holds unconditionally. The first part of Corollary \ref{CORO:gk_dihedral} is also proven in \cite{hIMC} when $p\neq 2$ using the language of Selmer groups.
	
	We highlight in the last part of this article some interesting connections between non-cyclotomic analogues of the Gross-Kuz'min conjecture and algebraic independence of $p$-adic logarithms of units of number fields. 
	
	Given an arbitrary $\Zp$-extension $K_\infty$ of $K$, we will define a map $\sL_{K_\infty/K}$ specializing to $\sL_K$ if $K_\infty$ is the cyclotomic extension of $K$. As noted in \cite{kisilevsky,jaulentsands}, there do exist examples of $\Zp$-extensions $K_\infty/K$ for which $\delta_{K_\infty/K}^\textbf{G}>0$, but a conjectural description of all such $K_\infty/K$ is still missing.
	
	In the next theorem, we fix an embedding $\ob{\bQ}\subset \ob{\bQ}_p$ and we let $\Lambda$ be the $\ob{\bQ}$-linear subspace of $ \ob{\bQ}_p$ generated by $1$ and by $p$-adic logarithms of non-zero algebraic numbers.
	
	\begin{THM}\label{THM:non-cyclotomic_gk}
		Let $k$ be an imaginary quadratic field, and $K$ an abelian extension of $k$ in which $p$ splits completely. 
		Then there exist at most finitely many distinct $\Zp$-extensions $k_\infty$ of $k$ for which $\delta_{Kk_\infty/K}^\textbf{G}>0$. Moreover, no such $\Zp$-extensions exist if the polynomial 
		\[XYZ^2-(AX-BY)(CX-DY)\in\bZ[A,B,C,D,X,Y,Z]\]
		does not vanish on any $7$-tuple $(a,b,c,d,x,y,z)\in\Lambda^7$ which form a $\ob{\bQ}$-linearly independent set.
	\end{THM}
	This last condition should be true according to the weak $p$-adic Schanuel conjecture. In Proposition \ref{prop:jaulentsands} we illustrate Theorem \ref{THM:non-cyclotomic_gk} with a classical application to the semi-simplicity of Iwasawa modules attached to $Kk_\infty/K$. 
	
	Theorem \ref{THM:non-cyclotomic_gk} can be generalized to arbitrary base fields $k$ having at most $r$ linearly disjoint $\Zp$-extensions with $r\leq 2$ (Theorem \ref{thm:non_cyclotomic_defects}). The main idea is that, under our assumption on $p$, one can parameterize $\Zp$-extensions of $k$ by points on a $(r-1)$-dimensional linear subspace $L$ of $\bP^{n-1}(\Qp)$, where $n=[k:\Q]$. The condition $\delta_{Kk_\infty/K}^\textbf{G}>0$ then cuts out a closed subvariety $\sC$ of $L$ given by polynomial equations with coefficients in $\Lambda_0:=\log_p(\cO_K[\tfrac{1}{p}]^\times)\subset \ob{\bQ}_p$. We then exploit the fact that any linear (resp. algebraic) independence between elements of $\Lambda_0$ implies strong conditions on the $\ob{\Q}$-points (resp. the $\ob{\Q}_p$-points) of $\sC$.
	
	Theorem \ref{THM:non-cyclotomic_gk} was inspired by Betina-Dimitrov's work \cite{betinadimitrovKatz} where the authors show the non-vanishing of a certain $\cL$-invariant for Katz's $p$-adic $L$-function restricted to the anticyclotomic $\Zp$-extension. In fact, their result generalizes to any $\Zp$-extension with non-transcendental slope. We expect that our techniques can give further results on the non-vanishing of $\cL$-invariants in more general contexts. 
	
	The paper is structured as follows. In Section \ref{sec:transcendence} we recall all the classical results in $p$-adic transcendence theory which we make use of. In Section \ref{sec:regulator} we describe Leopoldt's and Gross's defects via class field theory and we show that they are compatible with Artin formalism. Our main results and corollaries are proven in Section \ref{sec:bounds}, except for Theorem \ref{THM:non-cyclotomic_gk} whose proof is postponed to Section \ref{sec:vanishing_locus}.
	
	\subsection*{Acknowledgments}
	The author would like to thank Dominik Bullach and Lassina Demb\'el\'e for reading and providing comments on a preliminary version of this paper. This research is supported by the Luxembourg National Research Fund, Luxembourg, INTER/
	ANR/18/12589973 GALF.

	\section{$p$-adic transcendence theory}\label{sec:transcendence}
	Throughout this section we fix an embedding $\iota_p \colon \ob{\bQ}\hookrightarrow\ob{\bQ}_p$, allowing us to view algebraic numbers as $p$-adic numbers. The following very strong conjecture describes the algebraic dependence between logarithms of algebraic numbers (see \cite[Conjecture 3.10]{mazurcalegari}).
	
	\begin{conjecture*}[Weak $p$-adic Schanuel conjecture]
		Let $\alpha_1,\ldots,\alpha_n$ be nonzero algebraic numbers. If $\log_p(\alpha_1),\ldots,\log_p(\alpha_n)$ are linearly independent over $\bQ$, then they are algebraically independent over ${\bQ}$. 
	\end{conjecture*}
	
	We recall some classical results of Brumer, Waldschmidt and Roy and deduce some consequences that turn out to be useful in the study of the Gross-Kuz'min conjecture.
	\subsection{The Baker-Brumer theorem}
	Brumer \cite{brumer} extended Baker's method to the $p$-adic setting and proved the following theorem on linear independence of logarithms.
	\begin{theorem}[Baker-Brumer theorem]
		Let $\alpha_1,\ldots,\alpha_n$ be nonzero algebraic numbers. If $\log_p(\alpha_1),\ldots,\log_p(\alpha_n)$ are linearly independent over $\bQ$, then they are linearly independent over $\ob{\bQ}$. 
	\end{theorem}
	Recall that $\log_p$ is normalized so that we have $\log_p(p)=0$.
	\begin{proposition}\label{prop:log_p_almost_injective}
		Let $H\subset \ob{\bQ}$ be a number field. The $\ob{\bQ}$-linear extension $\log_p \colon \ob{\bQ} \otimes_\bZ H^\times \rightarrow \ob{\bQ}_p$, $c\otimes x \mapsto c\log_p(\iota_p(x))$ of the $p$-adic logarithm has kernel the line $p^{\ob{\bQ}}$ spanned by $1\otimes p$.
	\end{proposition}
	\begin{proof}
		Let $H_p$ be the completion of $\iota_p(H)$ inside $\ob{\Q}_p$, let $\cO^\times_{H_p}$ be its unit group and consider the abelian group $\cT=\{x\in H^\times \ \colon\ \iota_p(x)\in \cO_{H_p}^\times\}$. Then we clearly have $\ob{\bQ}\otimes H^\times = \left(\ob{\bQ} \otimes \cT\right) \bigoplus p^{\ob{\bQ}}$. Moreover, the $p$-adic logarithm map is injective on $\cO_{H_p}^\times$ modulo the torsion, so multiplicatively independent numbers $\alpha_1,\ldots,\alpha_n\in \cT$ have $\ob{\bQ}$-linearly independent $p$-adic logarithms by the Baker-Brumer theorem. This shows that the restriction of $\log_p$ to $\ob{\bQ} \otimes \cT$ is injective, hence $\ker (\log_p) = p^{\ob{\bQ}}$.
	\end{proof}
	\subsection{Waldschmidt's and Roy's theorem}
	Recall that $\Lambda$ is the $\ob{\bQ}$-linear subspace of $ \ob{\bQ}_p$ generated by $1$ and by $p$-adic logarithms of non-zero algebraic numbers. 
	
	Extensions of Baker's method due to Waldschmidt and Roy give a lower bound for the rank of matrices with coefficients in $\Lambda$. To each matrix $M$ with coefficients in $\ob{\bQ}_p$, of size $m\times \ell$, they assign a number $\theta(M)$ defined as the minimum of all ratios $\frac{\ell'}{m'}$ where $(m',\ell')$ runs among the pairs of integers satisfying $0<m'\leq m$ and $0\leq \ell'\leq \ell$, for which there exist matrices $P\in \GL_m(\ob{\bQ})$ and $Q\in\GL_\ell(\ob{\bQ})$ such that the product $PMQ$ can be written as
	\[\begin{pmatrix}
		M' & 0 \\ N & M''
	\end{pmatrix}\]
	with $M'$ of size $m' \times \ell'$. Note that $\theta(M) \leq \tfrac{\ell}{m}$ with equality if all the entries of $M$ are $\ob{\bQ}$-linearly independent. The following theorem is Roy's sharpening of Waldschmidt theorem (\cite[Théorème 2.1.p]{waldschmidt}, \cite[Corollary 1]{roy}).
	\begin{theorem}\label{thm:waldschmidt_roy}
		Let $M$ be a matrix with coefficients in $\Lambda$, of size $m\times \ell$ with $m,\ell>0$, and let $n$ be its rank. We have
		$$n \geq \frac{\theta(M)}{1+\theta(M)}\cdot m.$$
	\end{theorem} 
	
	Roy also deduced a useful corollary for $3\times 2$ matrices from Theorem \ref{thm:waldschmidt_roy} in \cite[Corollary 2]{roy}.
	\begin{corollaire}[Strong six exponentials theorem]\label{coro:strong_six_exponentials}
		Let $M$ be a $(3\times 2)$-matrix with coefficients in $\Lambda$. If the rows of $M$ are $\ob{\Q}$-linearly independent, and if the columns of $M$ are also $\ob{\Q}$-linearly independent, then $M$ has rank 2.
	\end{corollaire}

	\section{Regulator maps and class groups}\label{sec:regulator}

	\subsection{Galois cohomology}\label{sec:galois_cohomology}
	For all fields $L\subset \ob{\bQ}$ and all finite sets $S$ of places of $L$ containing $S_p(L)$, we let $X(L)$ (resp. $X'_S(L)$) be the Galois group of the maximal abelian pro-$p$ extension of $L$ which is unramified everywhere  (resp. unramified everywhere and totally split at all $v\in S$). If $S=S_p(L)$, we simply put $X'(L)=X'_S(L)$. Given a $\Zp$-extension $K_\infty=\bigcup_n K_n$ of $K$ with Galois group $\Gamma$, we have $X(K_\infty)=\varprojlim_n X(K_n)$ and $X_S'(K_\infty)=\varprojlim_n X_S'(K_n)$, the transition maps being the restriction maps. Therefore, $X(K_\infty)$ and $X'_S(K_\infty)$ are modules over the Iwasawa algebra $\Zp[[\Gamma]]$. They are finitely generated torsion as shown by Iwasawa \cite{iwasawa1973zl}. We let 
	\[\delta^\textbf{G}_{K_\infty/K}:=\rk_{\Zp} X'(K_\infty)_\Gamma,\]
	where $(-)_\Gamma$ means $\Gamma$-coinvariants. If $K_{\infty}$ is the cyclotomic $\Zp$-extension $K_{\cyc}$ of $K$ we simply write $\delta^\textbf{G}_K$ for $\delta^\textbf{G}_{K_\infty/K}$. We will later see that this definition is compatible with that of the introduction. One motivation in classical Iwasawa theory to compute $\delta^\textbf{G}_{K_\infty/K}$ originates in the following simple result by Jaulent and Sands \cite[Proposition 6]{jaulentsands}.
	
	\begin{proposition}\label{prop:jaulentsands}
		Let $\gamma$ be a topological generator of $\Gamma$. If no $p$-adic prime of $K$ splits completely in $K_\infty$ and if $\delta^\textbf{G}_{K_\infty/K}=0$, then $\gamma-1$ acts semi-simply on $X(K_\infty)$. That is, $(\gamma-1)^2$ does not divide the elements $P_i\in \Zp[[\Gamma]]$ appearing in any elementary module $\bigoplus_i \Zp[[\Gamma]]/(P_i)$ pseudo-isomorphic to $X(K_\infty)$.
	\end{proposition}
	
	Let $S_0 \supseteq S_p(K)\bigcup S_\infty(K)$ be a finite set of places of $K$. For any extension $L$ of $K$ and any discrete (resp. compact) $G_L$-module $M$ which is unramified outside the places of $L$ above $S_0$, we consider for all $i\geq 0$ the $S_0$-ramified $i$-th cohomology group (resp. continuous cohomology group) $\HH^i_{S_0}(L,M)=\HH^i(\Gal(L_{S_0}/L),M)$, where $L_{S_0}/L$ is the largest extension of $L$ which is unramified outside the places of $L$ above ${S_0}$. Given any subset $S\subset {S_0}$, let
	$$\Sha_{S}^i(L,M) = \ker\left[ \HH^i_{{S_0}}(L,M) \longrightarrow \prod_{v\in S} \HH^i(L_v,M) \times \prod_{v\in S_0-S} \HH^i(L_v^\ur,M)\right],$$
	where $L_v^\ur/L_v$ denotes the maximal unramified extension of $L_v$ (so $\bR^\ur=\bR$ in particular) and the maps above are the usual localization maps. Note that the definition of $\Sha_{S}^i(L,M)$ does not depend on the choice of $S_0$. We simply write $\Sha^i(L,M)$ instead of $\Sha_{S}^i(L,M)$ if $S=S_p(L)\bigcup S_\infty(L)$. We also write $M^*$ for the Pontryagin dual $\Hom_{\Zp}(M,\Qp/\Zp)$ of a $\Zp$-module $M$.
	
	\begin{lemme}\label{lem:class_groups_and_sha}
		Let $L\subset \ob{\bQ}$ be a number field and let $S\supset S_p(L)\bigcup S_\infty(L)$ be a finite set of places of $L$. There are canonical isomorphisms $\Sha_S^2(L,\Zp(1)) \simeq \Sha_S^1(L,\Qp/\Zp)^* \simeq X'_S(L)$.
	\end{lemme}
	\begin{proof}
		The first isomorphism is given by the Poitou-Tate duality theorem \cite[Theorem 4.10 (a)]{milneADT}. Since $\HH^1 _S(L,\Qp/\Zp)=\Hom(\Gal(L_S/L),\Qp/\Zp)$, class field theory easily implies $\Sha^1(L,\Qp/\Zp)=\Hom(X'_S(L),\Qp/\Zp)$.
	\end{proof}
	
	The isomorphisms provided by Lemma \ref{lem:class_groups_and_sha} are functorial in $L$ in the sense that, given a finite extension $L'/L$ of number fields, the norm map $X(L') \rightarrow X(L)$ corresponds to the corestriction map (resp. to the Pontryagin dual of the restriction map) $\Sha^2(L',\Z_p(1)) \rightarrow \Sha^2(L,\Zp(1))$ (resp. $\Sha^1(L',\Qp/\Zp)^* \rightarrow \Sha^1(L,\Qp/\Zp)^*$). 
	
	Given any $\Zp$-extension $K_\infty=\bigcup_n K_n$ of $K$ and any finite set $S \supset S_p(K)\bigcup S_\infty(K)$, Lemma \ref{lem:class_groups_and_sha} provides isomorphisms of $\Zp[[\Gamma]]$-modules
	\[\varprojlim_n \Sha_{S}^2(K_n,\Zp(1)) \simeq \Sha_{S}^1(K_\infty,\Qp/\Zp)^* \simeq X'_S(K_\infty).\]
	We now make use of the inflation-restriction exact sequence to study the problem of Galois descent. We have a commutative diagram with exact rows
	\small\begin{equation}
		\label{eq:diagramme_descente}
		\begin{tikzcd}
			0 \ar[r] & \Hom(\Gamma,\Qp/\Zp) \ar[r] \ar[d] & \HH^1_{\{p,\infty\}}(K,\Qp/\Zp) \ar[r] \ar[d] & \HH^1_{\{p,\infty\}}(K_\infty,\Qp/\Zp)^\Gamma \ar[r] \ar[d] & 0  \\
			0 \ar[r] & \bigoplus_v\Hom(\Gamma_v,\Qp/\Zp) \ar[r] & \bigoplus_v\HH^1(K_v,\Qp/\Zp) \ar[r] &  \bigoplus_w \HH^1(K_{\infty,w},\Qp/\Zp),
		\end{tikzcd}
	\end{equation} 
	\normalsize
	where the places $v$ (resp. $w$) of the second row run through all the $p$-adic and archimedean places of $K$ (resp. of $K_\infty$). Here, we have used the fact that $\Gamma$ has cohomological dimension one as it is pro-cyclic.
	
	\begin{proposition}\label{prop:premier_calcul_rang}
		We have $\delta^\textbf{G}_{K_\infty/K}=\dim \ker (\Loc_{K_\infty/K})$, where $\Loc_{K_\infty/K}$ is the localization map
		$$\Loc_{K_\infty/K} \colon \HH^1_{\{p\}}(K,\ob{\Q}_p)/\HH^1(\Gamma,\ob{\Q}_p) \longrightarrow \bigoplus_{\gP\in S_p(K)} \HH^1(K_\gP,\ob{\Q}_p)/\HH^1(\Gamma_\gP,\ob{\Q}_p).$$
	\end{proposition}
	\begin{proof}
		By Lemma \ref{lem:class_groups_and_sha}, the kernel of the right vertical map of (\ref{eq:diagramme_descente}) is equal to the Pontryagin dual of $X'(K_\infty)_\Gamma$. Since $\HH^1(K_v,\Qp/\Zp)$ for $v|\infty$ is finite, this implies that $\delta_{K_\infty/K}^\textbf{G}$ is equal to the rank of the Pontryagin dual of the kernel of the natural map 
		$$\HH^1_{\{p\}}(K,\Qp/\Zp)/\HH^1(\Gamma,\Qp/\Zp) \rightarrow \bigoplus_{\gP\in S_p(K)} \HH^1(K_\gP,\Qp/\Zp)/\HH^1(\Gamma_\gP,\Qp/\Zp).$$ 
		To end the proof, it suffices to notice that for $\cG=\Gal(K_{S}/K)$, $\Gal(\ob{K}_\gP/K_\gP)$,  $\Gamma$ or $\Gamma_\gP$, the natural map $\HH^1(\cG,\Zp) \otimes \Qp/\Zp \rightarrow\HH^1(\cG,\Qp/\Zp)$ (resp. $\HH^1(\cG,\Zp) \otimes \ob{\Q}_p\rightarrow \HH^1(\cG,\ob{\Q}_p)$) has finite kernel and cokernel (resp. is an isomorphism), see \cite[Appendix B \S~2]{rubinES}.
	\end{proof}
	
	\begin{remark}
		Since $\HH^1_{\{p\}}(K,\ob{\Q}_p)=\Hom(G_K,\ob{\Q}_p)$ parameterizes the $\Zp$-extensions of $K$, the domain of $\Loc_{K_\infty/K}$ has dimension $r_2 +\delta^\textbf{L}_K$, where $r_2$ is the number of complex places of $K$ (\cite[Theorem 13.4]{washington1997introduction}). In particular, Proposition \ref{prop:premier_calcul_rang} yields an upper bound $\delta^\textbf{G}_{K_\infty/K} \leq r_2 + \delta^\textbf{L}_K$. Therefore, Leopoldt's conjecture for a totally real field $K$ implies the Gross-Kuz'min conjecture for $K$, as already noticed by Kolster in \cite[Corollary 1.3]{kolster}.
	\end{remark}
	
	Given a prime $\gP\in S_p(K)$, let $\Gamma_{\gP}$ be the decomposition subgroup of $\Gamma$ at $\gP$ and denote by $\rec_{\Gamma_{\gP}} \colon K_\gP^\times \to \Gamma_\gP$ the corresponding local reciprocity map. Define also the $\Zp$-module \[\cH_{K_\infty/K}:=\ker\left(\bigoplus_{\gP\in S_p(K)} \Gamma_\gP \longrightarrow \Gamma \right).\] 
	By the usual product formula in class field theory the regulator map
	\begin{equation}\label{eq:def_sL_Gamma}
		\sL_{K_\infty/K} \colon \left\{\begin{array}{ccl}
			\cO_K[\tfrac{1}{p}]^\times & \longrightarrow &\cH_{K_\infty/K} \\
			x & \mapsto & (\rec_{\Gamma_\gP}(x))_\gP,
		\end{array}\right.
	\end{equation}
	is well-defined, and it extends to a $\ob{\Q}_p$-linear map $\cO_K[\tfrac{1}{p}]^{\times,\wedge} \to \cH_{K_\infty/K}^\wedge$ which we still denote by $\sL_{K_\infty/K}$. If $K_\infty=K_{\cyc}$, then the character $\log_p \circ \chi_\cyc \circ \rec_{\Gamma_{\gP}} \colon K_\gP^\times \longrightarrow \Qp$ coincides with $- \log_p \circ \cN_{\gP}$, where $\chi_\cyc$ is the cyclotomic character. Therefore, $\sL_{K_{\cyc}/K}$ is essentially the same as the map $\sL_K$ of the introduction, and we easily see that $\delta_{K_{\cyc}/K}^\textbf{G}=\delta_K^\textbf{G}$.
	
	\begin{proposition}\label{prop:generalisation_jaulent}
		We have $\delta^\textbf{G}_{K_\infty/K}=\dim \coker (\sL_{K_\infty/K})$.
	\end{proposition}
	\begin{proof}
		By Kummer theory and local class field theory, Tate's local pairing $\HH^1(K_\gP,\ob{\Q}_p) \times \HH^1(K_\gP,\ob{\Q}_p(1))\rightarrow \ob{\Q}_p$ can be identified with the evaluation map $\Hom(K_\gP^\times,\ob{\Q}_p) \times K_\gP^{\times,\wedge} \rightarrow \ob{\Q}_p$. Therefore, the orthogonal complement of $\HH^1(\Gamma_\gP,\ob{\Q}_p)\subseteq \HH^1(K_\gP,\ob{\Q}_p)$ is simply the kernel of $\rec_{\Gamma_\gP} \colon K_\gP^{\times,\wedge} \twoheadrightarrow \Gamma^\wedge_\gP$. Using the fact that $\Sha^1(K,\Qp)=0$, Poitou-Tate's duality then yields an isomorphism
		{\small 
			\[\ker\left(\HH^1_{\{p\}}(K,\ob{\Q}_p)\to \bigoplus_{\gP\in S_p(K)} \frac{\HH^1(K_\gP,\ob{\Q}_p)}{\HH^1(\Gamma_\gP,\ob{\Q}_p)}\right) \simeq \coker \left(\oplus_\gP \rec_{\Gamma_\gP}:\  \cO_K[\tfrac{1}{p}]^{\times,\wedge} \to \bigoplus_{\gP\in S_p(K)} \Gamma^\wedge_\gP\right).
			\]
			\normalsize}
		In particular, $\dim \ker(\Loc_{K_\infty/K})+1=\dim \coker(\sL_{K_\infty/K})+1$, so Proposition \ref{prop:premier_calcul_rang} yields the desired equality.
	\end{proof}
	
	\subsection{Isotypic components}\label{sec:isotypic}
	We consider in this paragraph the situation where the $\Zp$-extension $K_\infty/K$ comes from the $\Zp$-extension $k_\infty/k$ of a subfield $k$ of $K$, which means that $K_\infty=Kk_\infty$. Assume that $K/k$ is Galois with Galois group $G$. 
	Given an algebraically closed field $\mathbf{Q}$ of characteristic zero (typically, $\ob{\bQ}$ or $\ob{\bQ}_p$), the $\textbf{Q}$-valued representations of $G$ are semi-simple and the regular representation of $G$ splits as
	$$\textbf{Q}[G] = \bigoplus_{\rho} e(\rho)\cdot \textbf{Q}[G] = \bigoplus_\rho W^{\oplus \dim \rho},$$
	where $(W,\rho)$ runs through the set of all the $\textbf{Q}$-valued irreducible representations of $G$ and 
	\begin{equation}\label{eq:idempotent}
	e(\rho)= \frac{\dim \rho}{|G|} \cdot \sum_{g\in G} \Tr(\rho(g^{-1}))g \in \textbf{Q}[G]
	\end{equation}
	is the usual idempotent attached to $\rho$. 
	
	For any finite set of places $S$ of $k$ containing $S_\infty(k)$, let $\cO_K[1/S]^\times$ be the group of $S$-units of $K$. Dirichlet's unit theorem implies that we have a decomposition of $\textbf{Q}[G]$-modules
	$$\textbf{Q}\otimes_\bZ \cO_K[1/S]^\times = \left(\textbf{Q}\otimes_{\Z} \cO_k[1/S]^\times\right) \bigoplus\left( \bigoplus_{\mathds{1}\neq\rho} W^{\oplus d^+_{S}(\rho)}\right),$$
	where $(W,\rho)$ runs through the set of all non-trivial irreducible representations of $G$ and $d^+_{S}(\rho)= \sum_{v\in S}\dim \HH^0(k_v,W)$. It will be convenient to introduce the following invariants: 
	\begin{equation}\label{eq:def_invariants_d+_et_f}
		d(\rho):=[k:\Q]\cdot \dim \rho, \qquad	d^+(\rho):= \sum_{v|\infty} \dim \HH^0(k_v,W), \qquad f(\rho):=\sum_{\gp|p} \dim \HH^0(k_\gp,W),
	\end{equation}
	so that $d^+(\rho)=d^+_{S_\infty(k)}(\rho)$ and $f(\rho)=d^+_{S_\infty(k)\bigcup S_p(k)}(\rho)-d^+_{S_\infty(k)}(\rho)$. 
	
	We record in the next lemma a list of useful properties satisfied by the invariants introduced in (\ref{eq:def_invariants_d+_et_f}) and which we make use of in Sections \ref{sec:bounds} and \ref{sec:vanishing_locus}. Recall that a rule $\rho \mapsto a(\rho)\in \Z$, where $\rho$ runs among all the representations of Galois groups of finite extensions of number fields, is said to be compatible with Artin formalism if, for all finite Galois extensions $M/L/E$:
	\begin{enumerate}
		\item[(a)] $a(\tilde{\rho})=a({\rho})$ if $\tilde{\rho}\in \Art_{\ob{\bQ}_p}(\Gal(M/E))$ is the inflation of $\rho\in \Art_{\ob{\bQ}_p}(\Gal(L/E))$,
		\item[(b)] $a(\rho_1 \oplus \rho_2)=a(\rho_1)+a(\rho_2)$ for all $\rho_1,\rho_2 \in \Art_{\ob{\bQ}_p}(\Gal(M/E))$, and
		\item[(c)] $a(\rho')=a(\rho)$ if $\rho'$ is the induction of $\rho$ from $\Gal(M/L)$ to $\Gal(M/E)$.
	\end{enumerate}
	
	\begin{lemme}\label{lem:properties_invariants_d_d+_and_f}
		Let $L/E$ be a finite extension of number fields and let $a\in \{d,d^+,f\}$.
		\begin{enumerate}
			\item The rule $\rho \mapsto a(\rho)$ is compatible with Artin formalism.
			\item We have $d(\mathds{1}_L)=[L:\Q]$, $d^+(\mathds{1}_L)=|S_\infty(L)|$ and $f(\mathds{1}_L)=|S_p(L)|$. For any representation $(W,\rho)\in \Art_{\ob{\bQ}_p}(\Gal(M/L))$ with $\mathds{1}_L \not\subset \rho$, we have $\dim \Hom_{\Gal(M/L)}(W,\cO_M^{\times,\wedge})=d^+(\rho)$ and  $\dim \Hom_{\Gal(M/L)}(W,\cO_M[\tfrac{1}{p}]^{\times,\wedge})=d^+(\rho)+f(\rho)$.
			\item Let $\rho\in \Art_{\ob{\bQ}_p}(G_L)$ and let $M=\ob{\Q}^{\ker\rho}$ be the field extension cut out by $\rho$. Then $d^+(\rho)=d(\rho)$ if and only if $M$ is totally real. If $d^+(\rho)=0$, then $L$ is totally real and $M$ is a CM field. If $M$ has at least one real place, then $d^+(\rho) \geq \dim \rho$ and any subrepresentation $\theta$ of $\Ind_L^E \rho$ satisfies $d^+(\theta)\geq 1$. If $L$ has $r$ complex places, then $d^+(\rho)\geq r\cdot \dim \rho$.
			\item Let $\rho\in \Art_{\ob{\bQ}_p}(G_L)$. Then we have $a(\rho)\leq (\dim \rho)\cdot a(\mathds{1}_L)$. 
			\item Let $M/L$ be a finite Galois extension, let $\theta\in \Art_{\ob{\bQ}_p}(\Gal(M/L))$ be irreducible and let $\chi$ be a multiplicative character of $G_L$. Then we have $(\dim\theta)\cdot a(\theta \otimes \chi) \leq a(\chi_{|G_M})\leq a(\mathds{1}_M)$. Moreover, if $\theta\neq \mathds{1}_L$, then we have $(\dim\theta)\cdot a(\theta \otimes \chi) \leq a(\chi_{|G_M})-a(\chi)\leq a(\mathds{1}_M)-a(\chi)$.
		\end{enumerate}
	\end{lemme}
	\begin{proof}
		It is easy to deduce the claims (1), (2) and (4) from the definitions and Dirichlet's unit theorem. Let us justify claim (3). It is clear that $d^+(\rho)\leq d(\rho)$, and from \eqref{eq:def_invariants_d+_et_f} we have $d^+(\rho)=d(\rho)$ (resp. $d^+(\rho)=0$) if and only if all archimedean places of $L$ are real and $\rho(\sigma)=\id$ (resp. $\rho(\sigma)=-\id$) for all complex conjugations $\sigma \in \Gal(M/L)$. Since $\rho$ is faithful on $\Gal(M/L)$, this is equivalent to $M$ being totally real (resp. this implies $M$ being CM and $L$ totally real). 
		
		Assume now that $M$ has at least one real place $w$ and let $v$ (resp. $v_0$) be the place of $L$ (resp. of $E$) lying below $w$. Then we clearly have $d^+(\rho)\geq \dim\HH^0(L_v,\rho)=\dim\rho$. Moreover, if $\theta \subset \Ind_L^E \rho$, then the Frobenius reciprocity implies that there exists a subrepresentation $\rho' \subset \rho$ such that $\rho' \subset \theta_{|G_L}$, yielding $d^+(\theta)\geq \dim \HH^0(E_{v_0},\theta)=\dim \HH^0(L_v,\theta)\geq \HH^0(L_v,\rho')=\dim \rho'\geq 1$ as claimed. Suppose finally that $L$ has $r$ complex places $v_1,\ldots,v_r$. Then $d^+(\rho)\geq \sum_{i=1}^{r} \dim \HH^0(L_{v_i},\rho)=r\cdot \dim \rho$, so this ends the proof of claim (3).
		
		We now prove claim (5). The upper bounds on $a(\chi_{|G_M})$ directly follow from claim (4), so we only prove the lower bounds. Since $\theta$ is irreducible, the representation $(\theta \otimes \chi)^{\oplus\dim \theta}$ (and even $(\theta \otimes \chi)^{\oplus\dim \theta} \oplus \chi$ if $\theta\neq\mathds{1}_L$)  occurs as a subrepresentation of $(\Ind_M^L \mathds{1}_M)\otimes \chi= \Ind_M^L \chi_{|G_M}$. Artin formalism then yields the lower bounds of claim (5), as $a(\Ind_M^L \chi_{|G_M})=a(\chi_{|G_M})$ and $a$ takes non-negative values.
	\end{proof}
	
	We now describe the isotypic components of the map $\sL_{K_\infty/K}$. For $g\in G$, $\gP\in S_p(K)$ and $\eta$ a place of $K_\infty$ above $\gP$, the map $K_\gP \to  K_{g(\gP)}$ (resp. $K_{\infty,\eta} \to K_{\infty,\tilde{g}(\eta)})$ induced by $g$ (resp. by a lift  $\tilde{g}\in \Gal(K_\infty/k)$ of $g$) is a field isomorphism which yields a left $G$-action $x \mapsto g(x)$ (resp. $\gamma \mapsto \tilde{g}\gamma \tilde{g}^{-1}$) on $\bigoplus_{\gP|p} K_{\gP}^\times$ and on $\bigoplus_{\gP|p} \Gamma_\gP$ respectively. This action also restricts to $\cH_{K_\infty/K}$, and $G$ acts trivially on the quotient $(\bigoplus_{\gP|p} \Gamma_\gP)/\cH_{K_\infty/K}$. Moreover, the map $\sL_{K_\infty/K}$ is $G$-equivariant for the natural $G$-action on $\cO_K[\tfrac{1}{p}]^{\times}$ and the action on $\cH_{K_\infty/K}$ described above.
	
	Fix any $(W,\rho)\in \Art_{\ob{\Q}_p}(G)$ and let $\Hom_G(X,Y)$ be the $\ob{\Q}_p$-vector space of all $G$-equivariant linear maps between two $\ob{\Q}_p[G]$-modules $X$ and $Y$. By definition, the $\rho$-isotypic component of a $G$-equivariant $\ob{\Q}_p$-linear map $f:X \to Y$ is the linear map $\Hom_G(W,X) \to \Hom_G(W,Y)$ obtained by post-composing by $f$. Write $\sL_{k_\infty/k}(\rho)$ for the $\rho$-isotypic component of $\sL_{K_{\cyc}/K}$ and define
	\[\delta_{k_\infty/k}^\textbf{G}(\rho) := \dim \coker\left( \sL_{k_\infty/k}(\rho)\right).\]
	If $k_\infty=k_{\cyc}$, we abbreviate $\sL_{k_\infty/k}(\rho)$ and $\delta_{k_\infty/k}^\textbf{G}(\rho)$ as $\sL_k(\rho)$ and $\delta^\textbf{G}_k(\rho)$ respectively.
	
	For all $\gp\in S_p(k)$, fix a place $\gP_0$ of $K$ above $\gp$, let $G_\gp$ be the decomposition subgroup of $G$ at $\gP_0$, let $W_{\gp}^0=W^{G_\gp}$ and denote by $\res_\gp$ the restriction-to-$W_\gp^0$ map.
	
	\begin{proposition} \label{prop:identification_sL_rho_composante}
		The map $\sL_{k_\infty/k}(\mathds{1})$ can be naturally identified with $\sL_{k_\infty/k}$. If $\mathds{1}\not \subset \rho$, then the map $\sL_{k_\infty/k}(\rho)$ can be naturally identified with the composite map 
		{\small 
			\begin{equation}\label{eq:definition_sL_rho_Gamma}
				\begin{tikzcd}
					\Hom_G(W,\cO_K[\tfrac{1}{p}]^{\times,\wedge})  \ar[r,"\oplus \loc_\gp"] &  \bigoplus_{\gp} \Hom_{G_\gp}(W, K_{\gP_0}^{\times,\wedge})  \ar[r,"\oplus \res_\gp"] & \bigoplus_{\gp} \Hom(W_{\gp}^0, k_\gp^{\times,\wedge}) \ar[r,"\oplus \rec_{\Gamma_\gp}"] & \bigoplus_{\gp} \Hom(W_{\gp}^0,\Gamma_\gp^\wedge).
				\end{tikzcd}
			\end{equation}\normalsize}
		Here, $\gp$ runs over $S_p(k)$ in each sum, and we implicitly used the fact that $\left(K_{\gP_0}^{\times,\wedge}\right)^{G_\gp}=k_{\gp}^{\times,\wedge}$.
	\end{proposition}
	\begin{proof}
		Let $j \colon \cH_{K_\infty/K} \hookrightarrow \bigoplus_{\gP|p} \Gamma_\gP$ be the inclusion map and let $j(\rho)$ be its $\rho$-isotypic component. 
		
		Given a prime $\gp\in S_p(k)$ and a fixed prime $\gP_0|\gp$ of $K$ as before, we have $\bigoplus_{\gP|\gp} K_{\gP}^\times = \Ind_{G_\gp}^G K_{\gP_0}^\times$ and $\bigoplus_{\gP|\gp} \Gamma_\gP = \Ind_{G_\gp}^G \Gamma_{\gP_0}$ as $G$-modules and Frobenius reciprocity shows that $\Hom_G(W,\bigoplus_{\gP|\gp} K_\gP^{\times,\wedge})\simeq\Hom_{G_\gp}(W,K_{\gP_0}^{\times,\wedge})$ and $\Hom_G(W,\bigoplus_{\gP|\gp} \Gamma_\gP^{\times,\wedge})\simeq\Hom_{G_\gp}(W,\Gamma_{\gP_0}^{\times,\wedge})$, the isomorphisms being the natural projection maps. Therefore, $j(\rho)\circ \sL_{k_\infty/k}(\rho)$ can be identified with the composite map 
		\small$$\begin{tikzcd}
			\Hom_G(W,\cO_K[\tfrac{1}{p}]^{\times,\wedge})  \ar[r,"\oplus \loc_{\gp}"] & \bigoplus_{\gp|p} \Hom_{G_\gp}(W, K_{\gP_0}^{\times,\wedge}) \ar[r,"\oplus \rec_{\Gamma_{\gP_0}}"] & \bigoplus_{\gp|p} \Hom_{G_\gp}(W,\Gamma_{\gP_0}^\wedge)=\bigoplus_{\gp|p} \Hom(W_{\gp}^0,\Gamma_{\gP_0}^\wedge),
		\end{tikzcd}$$\normalsize
		where the last identification is induced by $\oplus \res_\gp$. Note that $\res_\gp$ and $\rec_{\Gamma_{\gP_0}}$ commute. Furthermore, letting $[n_\gp]:\Gamma_{\gP_0}^{\wedge} \rightarrow \Gamma_{\gp}^{\wedge}$ be the multiplication by $n_\gp=[K_{\gP_0}:k_\gp]$ map, the functoriality of Artin's reciprocity law shows that $\rec_{\gp}$ coincides with $[n_\gp]\circ \rec_{\gP_0}$ on $k_{\gp}^{\times,\wedge}$. Hence, if $\mathds{1} \not\subseteq \rho$, then the map $j(\rho) \circ\sL_{k_\infty/k}(\rho)$ coincides with the map of (\ref{eq:definition_sL_rho_Gamma}) under the identification $[n_\gp]:\Gamma_{\gP_0}^{\wedge} \simeq \Gamma_{\gp}^{\wedge}$. Since $j(\rho)$ is an isomorphism in this case, we obtain the desired description of $\sL_{k_\infty/k}(\rho)$. In the case where $\rho=\mathds{1}$, the map $\sL_{K_\infty/K}(\mathds{1})$ is nothing but the restriction  $(\sL_{K_\infty/K})^G \colon \cO_k[\tfrac{1}{p}]^{\times,\wedge} \to (\cH_{K_\infty/K}^\wedge)^G$ of $\sL_{K_\infty/K}$ to the $G$-invariants. Under the identifications
		$$(\bigoplus_{\gP|p} \Gamma_\gP^\wedge)^G = (\bigoplus_{\gp|p}\Ind_{G_\gp}^G \Gamma_{\gP_0}^\wedge )^G = \bigoplus_{\gp|p} \Gamma_{\gP_0}^\wedge \simeq \bigoplus_{\gp|p} \Gamma_\gp^\wedge$$ 
		induced by by Frobenius reciprocity and by $\oplus_\gp [n_\gp]$, it is clear that $(\cH_{K_\infty/K}^\wedge)^G$ is mapped onto $\cH_{k_\infty/k}^\wedge$, and that $\sL_{K_{\cyc}/K}(\mathds{1})$ can be  naturally identified with $\sL_{k_\infty/k}$.
		
	\end{proof}
	\begin{remark}
		The map $\sL_{k_\infty/k}(\rho)$ admits a more intrinsic description in terms of Bloch-Kato Selmer groups. Namely, let $\HH^1_{\ff,p}(k,\widecheck{W})\subset \HH^1(k,\widecheck{W})$ be the Selmer group of $\widecheck{W}$ defined by the Bloch-Kato condition at all places not dividing $p$ (see \cite[\S~3]{blochkato}), and let $\HH^1_\Gamma(k_\gp,\widecheck{W}^0_\gp)$ be the orthogonal complement of $\HH^1(\Gamma_\gp,W^0_\gp)\subset \HH^1(k_\gp,W)$ under Tate's local pairing. Then Kummer theory and the Inflation-Restriction exact sequence provide natural isomorphisms $\HH^1_{\ff,p}(k,\widecheck{W})\simeq \Hom_{G_k}(W,\cO_K[\tfrac{1}{p}]^{\times,\wedge})$ and $\HH^1(k_\gp,\widecheck{W})/\HH^1_\Gamma(k_\gp,\widecheck{W}^0_\gp)\simeq \Hom(W_\gp^0,\Gamma_\gp)$. An easy adaptation of the proof of Proposition \ref{prop:generalisation_jaulent} then shows that $\sL_{k_\infty/k}(\rho)$ coincides with the localization map $\HH^1_{f,p}(k,\widecheck{W}) \rightarrow \bigoplus_{\gp|p} \HH^1(k_\gp,\widecheck{W})/\HH^1_\Gamma(k_\gp,\widecheck{W}^0_\gp)$ under these identifications.
	\end{remark}
	\begin{corollaire}\label{coro:formule_gross_defect_over_Q}
		\begin{enumerate}
			\item 
			If $\rho=\mathds{1}$ then we have $\delta^\textbf{G}_{k_\infty/k}(\mathds{1})=\delta^\textbf{G}_{k_\infty/k}$. 
			\item Assume $k=\bQ$, $k_\infty=\bQ_{\cyc}$ and $\mathds{1}\not \subset \rho$. Fix an embedding $\iota_p \colon \ob{\Q} \hookrightarrow \ob{\Q}_p$ and let $W_p^0=W^{G_{\Qp}}$. Then,
			\begin{equation*}
				\delta^\textbf{G}_\Q(\rho)=\dim \coker\left[\Hom_G(W,\cO_K[\tfrac{1}{p}]^{\times,\wedge})  \longrightarrow \Hom(W_p^0,\ob{\Q}_p) \right],
			\end{equation*}
			the map being the restriction-to-$W_p^0$ map followed by the post-composition by $\log_p \circ \iota_p$.
		\end{enumerate} 
	\end{corollaire}
	\begin{proof}
		The first claim directly follows from Propositions \ref{prop:generalisation_jaulent} and \ref{prop:identification_sL_rho_composante}. The second claim follows from from Proposition \ref{prop:identification_sL_rho_composante} and from the fact that, if $\Gamma_p=\Gal(\Q_{p,\cyc}/\Qp)$, then the composite map $\log_p \circ \chi_{\cyc} \circ \rec_{\Gamma_p} \colon \Q_p^{\times,\wedge} \to \Gamma_p^\wedge \simeq \ob{\Q}_p \otimes_{\Zp} (1+p\Zp) \simeq \ob{\Q}_p$ coincides with $-\log_p$.
	\end{proof}
	
	\begin{corollaire}\label{coro:artin_formalism}
		The assignment $\rho \mapsto \delta^\textbf{G}_{k_\infty/k}(\rho)$ is compatible with Artin formalism. More precisely, 
		\begin{enumerate}
			\item[(a)] $\delta^\textbf{G}_{k_\infty/k}(\rho)$ does not depend on the choice of the splitting field $K$ of $\rho\in \Art_{\ob{\bQ}_p}(G_k)$.
			\item[(b)] $\delta^\textbf{G}_{k_\infty/k}(\rho_1\oplus \rho_2)=\delta^\textbf{G}_{k_\infty/k}(\rho_1)+\delta^\textbf{G}_{k_\infty/k}(\rho_2)$ for any $\rho_1,\rho_2\in \Art_{\ob{\bQ}_p}(G_k)$.
			\item[(c)] If $k_\infty=kk'_\infty$ for some subfield $k'\subset k$ and $\Zp$-extension $k'_\infty/k'$ and if $\rho'=\Ind_k^{k'}\rho$ is induced from $\rho\in \Art_{\ob{\bQ}_p}(G_k)$, then $\delta^\textbf{G}_{k'_\infty/k'}(\rho')=\delta^\textbf{G}_{k_\infty/k}(\rho)$.
		\end{enumerate}
	\end{corollaire}
	\begin{proof}
		Part (b) is obvious from the definition of $\sL_{k_\infty/k}(\rho)$. Part (a) is true if $\rho$ is trivial by Corollary \ref{coro:formule_gross_defect_over_Q} (1). Let $K'/K/k$ be Galois extensions, take $\mathds{1} \not\subseteq {\rho} \in \Art_{\ob{\Q}_p}(\Gal(K/k))$ and denote by $\tilde{\rho}$ its inflation to $\Gal(K'/k)$. Then the maps of (\ref{eq:definition_sL_rho_Gamma}) attached to $\rho$ and $\tilde{\rho}$ coincide on $\Hom_{\Gal(K/k)}(W,\cO_K[\tfrac{1}{p}]^{\times,\wedge})=\Hom_{\Gal(K'/k)}(W,\cO_{K'}[\tfrac{1}{p}]^{\times,\wedge})$, so $\delta^\textbf{G}_{k_\infty/k}(\rho)=\delta^\textbf{G}_{k_\infty/k}(\tilde{\rho})$. Hence, (a) holds for every $\rho\in \Art_{\ob{\Q}_p}(G_k)$. 
		
		As for (c), take any splitting field $K$ of $\rho$ which is Galois over $k'$ and put $G=\Gal(K/k)$, $G'=\Gal(K/k')$. Then Frobenius reciprocity identifies the $\rho'$-isotypic component of $\sL_{K_\infty/K}$ (seen as a $G'$-equivariant map) with the $\rho$-isotypic component of $\sL_{K_\infty/K}$ (seen as a $G$-equivariant map). Hence, the last property follows from Proposition \ref{prop:identification_sL_rho_composante}.
	\end{proof}
	We next define an analogous invariant for Leopoldt's conjecture. For any Galois extension $K/k$ with Galois group $G$, the localization map 
	\begin{equation*}
		\iota_K \colon \cO_K^{\times,\wedge} \longrightarrow \bigoplus_{\gP\in S_p(K)} \cO_{K_\gP}^{\times,\wedge}
	\end{equation*}
	is clearly $G$-equivariant. For any $(W,\rho)\in \Art_{\ob{\bQ}_p}(G)$ we let $\delta^\textbf{L}_k(\rho)$ be the dimension of the kernel of the $\rho$-isotypic component 
	\begin{equation}\label{eq:leopoldts_map_rho}
		\iota_k(\rho)\ \colon\ \Hom_G(W,\cO_K^{\times,\wedge}) \longrightarrow \bigoplus_{\gp\in S_p(k)} \Hom_G(W,\bigoplus_{\gP|\gp} \cO_{K_\gP}^{\times,\wedge}) \simeq \bigoplus_{\gp\in S_p(k)} \Hom_{G_\gp}(W,\cO_{K_{\gP_0}}^{\times,\wedge})
	\end{equation}
	of $\iota_K$. Here (and as in the definition of $\sL_{k_\infty/k}(\rho)$), $\gP_0$ is a fixed place of $K$ above $\gp$ for every place $\gp$ of $k$. The last isomorphism is given by Frobenius reciprocity and is induced by the natural projection map. As in Corollary \ref{coro:artin_formalism}, it is easy to see that the rule $\rho \mapsto \delta^\textbf{L}_k(\rho)$ is compatible with Artin formalism.
	\begin{remark}
		In terms of Bloch-Kato Selmer groups for $\widecheck{W}$, the injectivity of $\iota_k(\rho)$ is equivalent to that of the localization map $\HH^1_\ff(k,\widecheck{W}) \rightarrow \prod_{\gp|p} \HH^1_\ff(k_\gp,\widecheck{W})$. This last statement is Jannsen's conjecture for $\widecheck{W}$ (\cite{jannsen}).
	\end{remark}
	The following lemma on local Galois representations with finite image will help us describe $\delta^\textbf{L}_k(\rho)$ in another way if $k=\Q$. 
	\begin{lemme}\label{lem:fontaine}
		Let $(W,\rho)\in\Art_{\ob{\bQ}_p}(G_{\Qp})$ be a local representation factoring through the Galois group of a finite extension $E \subset \ob{\Q}_p$ of $\Qp$. Then the internal multiplication map $\ob{\Q}_p \otimes_{\Qp} E \rightarrow \ob{\Q}_p$ induces an isomorphism
		$$m \ \colon\ \Hom_{G_{\Qp}}(W,\ob{\Q}_p \otimes_{\Qp} E) \simeq \Hom_{\ob{\bQ}_p}(W,\ob{\Q}_p).$$
		Here, we let $G_{\Qp}$ act on $\ob{\Q}_p \otimes_{\Qp} E$ via $g(a\otimes x)=a\otimes g(x)$.
	\end{lemme}
	\begin{proof}
		Since $E\simeq \Qp[\Gal(E/\Qp)]$ as a Galois module, it is enough to show that $m$ is injective. Choose any finite Galois extension $L/\Qp$ which contains $E$ and over which $\rho$ is realizable, \textit{i.e.}, there exist a $L[G_{\Qp}]$-module $W_L$ and an isomorphism $W_L \otimes_L \ob{\Q}_p \simeq W$. Then we have to show that the map $\Hom_{L[G_{\Qp}]}(W_L, L \otimes_{\Qp} E) \to \Hom_{L}(W_L,L)$ is injective. By considering $\Qp$-linear homomorphisms instead of $L$-linear ones in the previous map, it suffices to prove that the map $(V \otimes_{\Qp} \ob{\Q}_p)^{G_{\Qp}} \otimes_{\Qp} \ob{\Q}_p \to V \otimes_{\Qp} \ob{\Q}_p$ given by $(v\otimes a)\otimes b \mapsto v \otimes (ab)$ is injective, where $V=\Hom_{\Qp}(W_L,\Qp)$ and $G_{\Qp}$ acts on both factors of $V \otimes_{\Qp} \ob{\Q}_p$. But $V$ is finite dimensional over $\Qp$, so this follows from the $\ob{\Q}_p$-admissibility in the sense of Fontaine of Galois representations of finite image (see e.g. \cite[Proposition 2.7]{fontaineouyang}).
	\end{proof}
	
	\begin{proposition}\label{prop:alternate_decription_leopoldt_defect}
		Let $(W,\rho)\in \Art_{\ob{\bQ}_p}(G_\Q)$, let $K\subset \ob{\Q}$ be a splitting field of $\rho$ and let $\gP_0$ be the $p$-adic place of $K$ defined by a fixed embedding $\iota_p \colon \ob{\Q} \hookrightarrow \ob{\Q}_p$. We have 
		\[\delta^\textbf{L}_\Q(\rho)=\dim \ker \left[\cL \ \colon\ \Hom_{G_{\Q}}(W,\cO_K^{\times,\wedge}) \rightarrow \Hom(W,\ob{\Q}_p)\right], \]
		the map $\cL$ being the post-composition by $\log_p  \colon \cO_K^{\times,\wedge} \to \ob{\Q}_p \otimes_{\Qp} K_{\gP_0}$ followed by the internal multiplication $\ob{\Q}_p \otimes_{\Qp} K_{\gP_0} \to \ob{\Q}_p$.
	\end{proposition}
	\begin{proof}
		First note that the $p$-adic logarithm $\log_p \colon \cO_E^\times \rightarrow E$ over a finite extension $E$ of $\Qp$ induces an isomorphism $\cO_E^{\times,\wedge} \simeq \ob{\bQ}_p \otimes_{\Qp} E$. Applying Lemma \ref{lem:fontaine} to $E=K_{\gP_0}$ we obtain isomorphisms 
		$\Hom_{G_{\Q_p}}(W,\cO_{K_{\gP_0}}^{\times,\wedge})\simeq \Hom_{G_{\Q_p}}(W,\ob{\Q}_p \otimes_{\Qp}K_{\gP_0})\simeq \Hom(W,\ob{\Q}_p)$. The map $\cL$ is simply the map $\iota_{\Q}(\rho)$ composed with these isomorphisms, so the claim follows.
	\end{proof}
	
	\section{Bounds on Leopoldt's and Gross's defects}\label{sec:bounds}
	Throughout this section we fix an embedding $\iota_p \colon \ob{\Q} \hookrightarrow \ob{\Q}_p$.
	\subsection{Bounds on Leopoldt's defect}
	\begin{theorem}\label{thm:leopoldt_main}
		Let $\rho\in \Art_{\ob{\bQ}_p}(G_\Q)$ be irreducible and let $d=d(\rho)$, $d^+=d^+(\rho)$. We have $\delta^\textbf{L}_\bQ(\rho) \leq \frac{(d^+)^2}{d+d^+}$.
	\end{theorem}
	\begin{proof}
		Since $\delta_\Q^\textbf{L}(\mathds{1})=\delta_\Q^\textbf{L}=0$, we may assume that $\rho\neq \mathds{1}$. Let $K$ be a splitting field of $\rho$ and let $G=\Gal(K/\Q)$. Recall from Lemma \ref{lem:properties_invariants_d_d+_and_f} that $\dim \Hom_G(W,\cO_K^{\times,\wedge})=d^+$. By Proposition \ref{prop:alternate_decription_leopoldt_defect} it is enough to show that the rank of the map $\Hom_G(W,\cO_K^{\times,\wedge})\rightarrow \Hom(W,\ob{\Q}_p)$ induced by $a\otimes x \mapsto a\log_p(\iota_p(x))$ on $\cO_K^{\times,\wedge}$ is at least $\frac{d\cdot d^+}{d+d^+}$. 
		
		Consider a $\ob{\bQ}$-structure on $W$, that is, a $\ob{\Q}$-linear representation $W_{\ob{\Q}}$ of $G$ such that $W_{\ob{\Q}}\otimes_{\ob{\bQ}} \ob{\Q}_p \simeq W$. Let also $w_1,\ldots,w_d$ be a $\ob{\Q}$-basis of $W_{\ob{\bQ}}$. We fix an isomorphism 
		$$e(\rho) \cdot \left(\ob{\Q} \otimes \cO_K^\times\right) \simeq W_{\ob{\Q}}^{\oplus d^+},$$ 
		where $e(\rho)$ is the idempotent attached to $\rho$ defined in \eqref{eq:idempotent}. Using this isomorphism, we may define $G$-equivariant morphisms $\Psi_1,\ldots,\Psi_{d^+} \colon W_{\ob{\Q}} \longrightarrow \ob{\Q} \otimes \cO_K^\times$ which form a basis of $\Hom_G(W,\cO_K^{\times,\wedge})$. Moreover, the elements $\Psi_j(w_i) \in \ob{\bQ} \otimes \cO_K^\times$ for $1\leq i \leq d$, $1 \leq j \leq d^+$ are $\ob{\bQ}$-linearly independent by construction, as well as their $p$-adic logarithms by Proposition \ref{prop:log_p_almost_injective}. Hence, the matrix $M=(\log_p(\iota_p(\Psi_j(w_i))))_{i,j}$ of size $d\times d^+$ satisfies $\theta(M)=\frac{d^+}{d}$ and Theorem \ref{thm:waldschmidt_roy} implies $\rk M \geq \frac{d\cdot d^+}{d+d^+}$ as claimed.
	\end{proof}
	\begin{corollaire}\label{coro:waldschmidt_bound}
		\begin{enumerate}
			\item Let $\rho \in \Art_{\ob{\bQ}_p}(G_k)$ and let $d^+=d^+(\rho)$. We have $\delta_k^\textbf{L}(\rho) \leq d^+/2$. 
			\item For every finite extensions $K/k$ of number fields, we have $\delta^\textbf{L}_K \leq \delta^\textbf{L}_k + \left(\rk \cO_K^\times - \rk \cO_k^\times \right)/2$. 
		\end{enumerate} 
	\end{corollaire}
	\begin{proof}
		Since $d^+\leq d$, the first inequality obviously follows from Theorem \ref{thm:leopoldt_main} for irreducible $\rho\in \Art_{\ob{\bQ}_p}(G_\bQ)$, hence it follows for general $\rho$ from Artin formalism. By Lemma \ref{lem:properties_invariants_d_d+_and_f} (2), the unique representation $\rho_0$ of $G_k$ such that $\Ind_K^k \mathds{1}=\mathds{1} \oplus \rho_0$ satisfies $d^+(\rho_0)=\rk \cO_K^\times - \rk \cO_k^\times$, so the second inequality follows from the first one applied to $\rho_0$.
	\end{proof}
	\subsection{Bounds on Gross's defect}
	\begin{theorem}\label{thm:g_k_main}
		Let $\rho\in \Art_{\ob{\bQ}_p}(G_\Q)$ be irreducible and let $f=f(\rho)$, $d^+=d^+(\rho)$. If $d^+=f=0$, then $\delta_\Q^\textbf{G}(\rho)=0$. Otherwise, we have $\delta_\bQ^\textbf{G}(\rho) \leq \frac{f^2}{d^++2f}$.
	\end{theorem}
	\begin{proof}
		Recall from Section \ref{sec:isotypic} that $\dim \Hom_G(W,\cO_K[\tfrac{1}{p}]^{\times,\wedge})=d^++f$ and that $\dim W_p^0=f$. If $\rho=\mathds{1}$ or $f=0$, then the codomain of $\sL_\Q(\rho)$ is $\{0\}$, yielding $\delta_\Q^\textbf{G}(\rho)=0$. We assume henceforth that $\rho\neq \mathds{1}$ and $f>0$. Let $K$ be a splitting field of $\rho$ and let $G=\Gal(K/\Q)$. By Corollary \ref{coro:formule_gross_defect_over_Q} (2) and Proposition \ref{prop:identification_sL_rho_composante} it is enough to show that the rank of the map $\Hom_G(W,\cO_K[\tfrac{1}{p}]^{\times,\wedge})\rightarrow \Hom(W_p^0,\ob{\Q}_p)$ induced by $a\otimes x \mapsto a\log_p(\iota_p(x))$ on $\cO_K[\tfrac{1}{p}]^{\times,\wedge}$ is at least $\frac{(d^++f)\cdot f}{d^++2f}$.
		
		As in the proof of Theorem \ref{thm:leopoldt_main}, fix a $\ob{\Q}$-structure $W_{\ob{\Q}}$ of $W$. Fix also a basis $w_1,\ldots,w_f$ of the subspace $W_{\ob{\Q},p}^0$ of $G_{\Qp}$-invariants of $W_{\ob{\Q}}$ and an isomorphism 
		$$e(\rho)\cdot \left(\ob{\Q} \otimes \cO_K[\tfrac{1}{p}]^\times\right) \simeq W_{\ob{\Q}}^{\oplus (d^++f)}.$$
		 These choices yield a basis $\Psi_1,\ldots,\Psi_{d^++f}$ of $\Hom_G(W,\cO_K[\tfrac{1}{p}]^{\times,\wedge})$ such that the elements $\Psi_j(w_i) \in \ob{\bQ} \otimes \cO_K^\times$ for $1\leq i \leq f$, $1 \leq j \leq d^++f$ are $\ob{\bQ}$-linearly independent. Since $e(\rho)$ kills $p^{\ob{\Q}}$, we deduce from Proposition \ref{prop:log_p_almost_injective} that the entries of the matrix $M'=(\log_p(\iota_p(\Psi_j(w_i))))_{i,j}$ of size $f \times (d^++f)$ are $\ob{\Q}$-linearly independent as well. Therefore, $\theta(M')=\tfrac{d^++f}{f}$ and Theorem \ref{thm:waldschmidt_roy} implies $\rk M' \geq \tfrac{(d^++f)\cdot f}{d^++2f}$.
	\end{proof}
	\begin{corollaire}\label{coro:gk_bound}
		\begin{enumerate}
			\item Let $\rho \in \Art_{\ob{\bQ}_p}(G_k)$, let $K$ be the field cut out by $\rho$ and let $f=f(\rho)$. Then the following inequalities hold true:
			\[\delta_k^\textbf{G}(\rho)\ \left\{\begin{array}{ccl}
				\leq &f/2 &  \\
				< &f/2 & \mbox{if $f\neq 0$ and $K$ contains at least one real place}\\
				\leq &f/3 & \mbox{if $K$ is totally real}
			\end{array}\right.
			\]
			\item Let $K/k$ be a finite extension of number fields. Then the following inequalities hold true:
			\[\delta_K^\textbf{G}\ \left\{\begin{array}{ccl}
				\leq & \delta_k^\textbf{G} + \left(|S_p(K)|-|S_p(k)|\right))/2 &  \\
				< & \delta_k^\textbf{G} + \left(|S_p(K)|-|S_p(k)|\right))/2 & \mbox{if $|S_p(K)|\neq|S_p(k)|$ and $K$ contains at least one real place}\\
				\leq & \delta_k^\textbf{G} + \left(|S_p(K)|-|S_p(k)|\right))/3 & \mbox{if $K$ is totally real}
			\end{array}\right.
			\]
		\end{enumerate} 
	\end{corollaire}
	\begin{proof}
		We know that $\rho \mapsto \delta_k^\textbf{G}(\rho)$ is compatible with Artin formalism by Corollary \ref{coro:artin_formalism}. We now explain how to prove (1). Again by Artin formalism, it suffices to prove (1) with $\rho$ replaced by any irreducible subrepresentation $\theta \in \Art_{\ob{\bQ}_p}(G_\Q)$ of $\Ind_k^\Q \rho$. For such a $\theta$, Lemma \ref{lem:properties_invariants_d_d+_and_f} implies that $d^+(\theta)\geq 1$ if $K$ contains at least one real place, and $d^+(\theta)=d(\theta)$ if $K$ is totally real. Therefore, the inequalities in (1) for $\theta$ directly follow from Theorem \ref{thm:g_k_main}. Finally, (2) follows from (1) as in the proof of Corollary \ref{coro:waldschmidt_bound}.
	\end{proof}
	\begin{remark}\label{rem:schanuel_implique_leo_et_gk}
		The matrices $M$ and $M'$ appearing in the course of the proof of Theorems \ref{thm:leopoldt_main} and \ref{thm:g_k_main} have full rank under the $p$-adic Schanuel conjecture. Therefore, Artin formalism shows that Leopoldt's and Gross-Kuz'min's conjectures hold in great generality under the $p$-adic Schanuel conjecture. 
	\end{remark}
	
	\subsection{Applications}
	\begin{theorem}
		Let $k$ be a totally real number field and let $(V,\rho)\in \Art_{\ob{\Q}_p}(G_k)$ be such that $d^+(\rho)=0$. Then Gross's $p$-adic regulator matrix $R_p(V)$ \cite[(2.10)]{gross1981padic} is of size $f(\rho)$ and of rank at least $ f(\rho)/2$.
	\end{theorem}
	\begin{proof}
		Let $K$ be the CM field cut out by $\rho$ and let $(\Hom_{\ob{\Q}_p}(V,\ob{\Q}_p),\rho^*)$ be the contragredient representation of $\rho$. Gross's regulator map $\lambda_p$ defined in \cite[(1.18)]{gross1981padic} can be identified with the ``minus part'' of $\sL_K$, which is, by definition, the restriction of $\sL_K$ to the subspace where the complex conjugation acts by $-1$. This means that $\lambda_p$ and $\sL_K$ share the same $\theta$-isotypic component for every representation $\theta\in \Art_{\ob{\bQ}_p}(\Gal(K/k))$ such that $d^+(\theta)=0$. Since taking $(V \otimes -)^{G_k}$ amounts to taking $\rho^*$-isotypic components, we conclude that $\rk R_p(V)=\rk \sL_k(\rho^*)=f(\rho^*)-\delta^\textbf{G}_k(\rho^*)$, so $\rk R_p(V)\geq f(\rho^*)/2=f(\rho)/2$ by Theorem \ref{thm:g_k_main}.
	\end{proof}
	
	In the next theorem, we write $k^+$ for the maximal totally real subfield of a number field $k$, and $\Q^\ab$ for the maximal abelian extension of $\Q$.
	\begin{theorem}\label{thm:list_new_cases_gk}
		Let $K/k$ be an abelian extension of number fields. The Gross-Kuz'min conjecture holds for $K$ in each of the following cases.
		\begin{enumerate}[label=(\alph*)]
			\item Either $|S_p(K)|\leq 2$, or $|S_p(K)|\leq 3$ and $K$ has at least one real place, or $|S_p(K)|\leq 4$ and $K/\Q$ is Galois, or $|S_p(K)|\leq 6$ and $K/\Q$ is a real Galois extension.
			\item $|S_p(k)|=1$, or $|S_p(k)|\leq 2$ and $K$ has at least one real place.
			\item $K\subset k\cdot \Q^\ab$, $k/\Q$ is Galois, and either $|S_p(k)|\leq 3$, or $|S_p(k)|\leq 5$ and $K$ is real.
			\item $k$ is an imaginary quadratic field, or $k$ is a real quadratic field and $K$ has at least one real place.
			\item $k/\Q$ is Galois, $|S_p(k)|\leq 2$, $|S_p(k^+)|=1$ and $[K:k]$ and $[k:\Q]$ are coprime.
		\end{enumerate}
	\end{theorem}
	\begin{proof}
		Recall that $\delta^\textbf{G}(-)$ is compatible with Artin formalism by Corollary \ref{coro:artin_formalism}. We shall often appeal to Lemma \ref{lem:properties_invariants_d_d+_and_f} and to the following consequence of Theorem \ref{thm:g_k_main} without further notice. For any irreducible representation $\theta \in \Art_{\ob{\bQ}_p}(G_\Q)$, we have $\delta_\Q^\textbf{G}(\theta)=0$ if $f(\theta)\leq 1$, or if $f(\theta)=2$ and $d^+(\theta)\geq 1$. In particular, $\delta_\Q^\textbf{G}(\theta)=0$ if $\theta$ is a multiplicative character of $G_\Q$, so $\delta_M^\textbf{G}=0$ for any abelian extension $M/\Q$ by Artin formalism. 
		
		Since $\delta_\Q^\textbf{G}=0$, it follows from Corollary \ref{coro:gk_bound} that $\delta_K^\textbf{G}=0$ for $K$ satisfying one of the two first assumptions in case (a). Consider the two last assumptions in (a) and assume that $K/\Q$ is Galois. We claim that $\delta_\Q^\textbf{G}(\theta)=0$ for all irreducible $\theta\in \Art_{\ob{\bQ}_p}(\Gal(K/\Q))$. We may assume that $\dim \theta\geq 2$, so $f(\theta)\leq (f(\mathds{1}_K)-f(\mathds{1}_\Q))/(\dim \theta)\leq (|S_p(K)|-1)/2$. The two last assumptions in (a) ensure that we either have $f(\theta)\leq 1$, or $f(\theta)\leq 2$ and $d^+(\theta)=d(\theta)\geq 1$, so we indeed have $\delta_\Q^\textbf{G}(\theta)=0$. Therefore, $\delta_K^\textbf{G}=0$ in case (a).
		
		Let $G=\Gal(K/k)$ and let $\hat{G}=\Hom(G,\ob{\Q}_p^\times)$ be the group of characters of $G$. We place ourselves in cases (b), (c) and (d), we fix ${\chi \in \hat{G}}$ and we show that $\delta_k^\textbf{G}(\chi)=0$. Since $f(\chi)\leq |S_p(k)|$, Corollary \ref{coro:gk_bound} (1) implies $\delta_k^\textbf{G}(\chi)=0$ in case (b). 
		Suppose now we are in case (c). Then $\chi$ descends to a character $\chi_\Q$ of $G_\Q$. Moreover, as $k/\Q$ is Galois, any irreducible subrepresentation $\rho$ of $\Ind_k^\Q \chi\simeq (\Ind_k^\Q \mathds{1}_k)\otimes\chi_\Q$ occurs $(\dim \rho)$ times, so it satisfies $f(\rho)\leq |S_p(k)|/(\dim \rho)$. Moreover, if $K$ is totally real, then any such $\rho$ satisfies $d^+(\rho)=d(\rho)\geq 1$, so we can conclude $\delta_\Q^\textbf{G}(\rho)=0$. Therefore, $\delta_k^\textbf{G}(\chi)=0$ in case (c).
		We now assume to be in case (d). Then $\Ind_k^\bQ \chi$ has dimension $2$, so it is either irreducible or it is the sum of two characters of $G_\Q$, say $\eta_1$ and $\eta_2$. In the latter case, we already know that $\delta_\bQ^\textbf{G}(\eta_i)=0$ for $i=1,2$, so $\delta_k^\textbf{G}(\chi)=\delta_\bQ^\textbf{G}(\Ind_k^\bQ \chi)=0$. If $\Ind_k^{\Q}\chi$ is irreducible, then the assumptions on $K$ and $k$ imply that $f(\chi)\leq2 $ and $d^+(\chi)\geq 1$, yielding $\delta_k^\textbf{G}(\chi)=0$. Therefore, $\delta_K^\textbf{G}=0$ in the cases (b), (c) and (d). 
		
		We now make the assumptions in (e) and we assume without loss of generality that $K$ is Galois over $\Q$ with Galois group $\cG$. By the Schur-Zassenhaus theorem, $\cG$ is the semi-direct product of $H:=\Gal(k/\Q)$ acting on $G$. Let $\rho\in \Art_{\ob{\bQ}_p}(\cG)$ be irreducible and let us prove that $\delta_\Q^\textbf{G}(\rho)=0$.
		By \cite[Chap II \S~8.2]{serrerepresentations}, $\rho$ can be written as $\Ind_{k'}^{\Q} (\theta\otimes\chi)$, where $k'/\Q$ is a subextension of $k/\Q$, $\theta$ an irreducible representation of $\Gal(k/k')$ and $\chi$ a character of $\Gal(K/k')$. Note that $f(\rho)\leq |S_p(k)|/(\dim\theta)\leq 2/(\dim\theta)$, so we may assume without loss of generality that $\dim \theta=1$.
		If $k'$ is totally real, then $f(\rho)\leq |S_p(k')|\leq|S_p(k^+)|=1$, and otherwise we have $d^+(\rho)\geq 1$. In any case, $\delta_\Q^\textbf{G}(\rho)=0$ so we may infer $\delta_K^\textbf{G}=0$ in case (e) as well.
	\end{proof}
	
	\section{Vanishing locus of Gross's defect}\label{sec:vanishing_locus}
	\subsection{Preliminaries}
	This section is devoted to the proof of the following theorem which, in turn, implies Theorem \ref{THM:non-cyclotomic_gk} stated in the introduction. 
	
	\begin{theorem}\label{thm:non_cyclotomic_defects}
		Let $k$ be a number field and let $\varphi \colon G_k \to \ob{\bQ}_p^\times$ be a finite-order character. Assume also that $p$ completely splits in the number field cut out by $\varphi$.
		\begin{enumerate}
			\item Assume that $r_2+\delta^\textbf{L}_k\leq 1$, where $r_2$ is the number of complex places of $k$. If there exists at least one $\Zp$-extension $k_\infty$ of $k$ such that $\delta_{k_\infty/k}^\textbf{G}(\varphi)=0$, then there are at most $[k:\Q]$ (and at most $[k:\Q]-1$ if $\varphi=\mathds{1}$) $\Zp$-extensions $k_\infty/k$ such that $\delta_{k_\infty/k}^\textbf{G}(\varphi)\neq 0$.
			\item Assume that $k$ is an imaginary quadratic field. There is at most one $\Zp$-extension $k_\infty/k$ for which $\delta_{k_\infty/k}^\textbf{G}(\varphi)\neq 0$, and it has a transcendental slope (see Example \ref{ex:slope} for a definition of the slope of $k_\infty/k$). Moreover, if $\varphi$ cuts out an abelian extension of $\Q$ or if the polynomial 
			$XYZ^2-(AX-BY)(CX-DY)$
			does not vanish on any $7$-tuple $(a,b,c,d,x,y,z)\in\Lambda^7$ which form a $\ob{\bQ}$-linearly independent set, then $\delta^\textbf{G}_{k_\infty/k}(\varphi)=0$ for any $\Zp$-extension $k_\infty$ of $k$. 
		\end{enumerate}
	\end{theorem}
	\begin{proof}[Proof of Theorem \ref{THM:non-cyclotomic_gk}, assuming Theorem \ref{thm:non_cyclotomic_defects}]
		Let $K$ be an abelian extension of an imaginary quadratic field $k$ and let $K^\ab \subset K$ be its maximal absolutely abelian subfield. By Artin formalism (Corollary \ref{coro:artin_formalism}) and by Theorem \ref{thm:non_cyclotomic_defects}, we have $\delta^\textbf{G}_{K k_\infty/K}>0$ for a given $\Zp$-extension $k_\infty/k$ if and only if there exists a character $\varphi$ of $\Gal(K/k)$ such that $\delta_{k_\infty/k}^\textbf{G}(\varphi)> 0$. Such a character cannot be a character of $\Gal(K^\ab/k)$, and moreover, $k_\infty$ is uniquely determined by $\varphi$. Therefore, we have $\delta^\textbf{G}_{K k_\infty/K}>0$ for at most $[K:k]-[K^\ab:k]$ distinct $\Zp$-extensions of $k$.
	\end{proof}
	
	In the rest of Section \ref{sec:vanishing_locus}, we fix once and for all an abelian extension $K/k$ with Galois group $G$ such that $p$ totally splits in $K$. We let $n$ be the degree of $k$ and $(r_1,r_2)$ its signature, and we put $r=r_1+r_2-1-\delta^\textbf{L}_k$ so that the maximal multiple $\Zp$-extension of $k$ has rank $n-r$. Write $\gp_1,\ldots,\gp_n$ for the $p$-adic primes of $k$. Finally, denote by $\sZ(k)$ the set of all $\Zp$-extensions of $k$. 
	
	Instead of working with the map $\sL_{k_\infty/k}(\varphi)$, it will be more convenient to consider the following alternative description of $\delta_{k_\infty/k}^\textbf{G}(\varphi)$.
	
	\begin{lemme}\label{lem:expression_alternative_non-cyc_defect}
		Let $k_\infty/k$ be a $\Zp$-extension with Galois group $\Gamma$. The quantity $\delta_{k_\infty/k}^\textbf{G}(\varphi)$ is the dimension of the kernel of the $\varphi$-isotypic component of the map $\Loc_{K_\infty/K}$ of Proposition \ref{prop:premier_calcul_rang}.
	\end{lemme}
	\begin{proof}
		This follows from Poitou-Tate duality as in the proof of Proposition \ref{prop:generalisation_jaulent} where one replaces the $G_K$-module $\Qp$ by the module $\ob{\Q}_p(\varphi)$ on which $G_k$ acts by $\varphi$.
	\end{proof}
	
	Since $\varphi$ is a multiplicative character, the $\varphi$-isotypic component of a $\ob{\Q}_p[G]$-module $X$ is canonically isomorphic to the linear subspace $X[\varphi]$ of $X$ consisting of elements $x\in X$ such that $g\cdot x=\varphi(g)x$ for all $g\in G$. Hence Lemma \ref{lem:expression_alternative_non-cyc_defect} asserts that, if $\varphi\neq \mathds{1}_k$, then $\delta_{k_\infty/k}^\textbf{G}(\varphi)$ is equal to the dimension of the kernel of the localization map
	\begin{equation}\label{eq:localization_phi_component}
		\Hom(G_K,\ob{\bQ}_p)[\varphi] \longrightarrow \bigoplus_{i=1}^n \left( \bigoplus_{\gP|\gp_i} \HH^1(K_{\gP},\ob{\bQ}_p)/\Hom(\Gamma_{\gP},\ob{\bQ}_p)\right)[\varphi].
	\end{equation}
	
	\subsection{Matrices in logarithms of algebraic numbers}
	For any $\gP\in S_p(K)$, we identify $\HH^1(K_\gP,\ob{\Q}_p)$ with $\Hom(K_{\gP}^\times,\ob{\Q}_p)\simeq \Hom(\Q_p^\times,\ob{\Q}_p)$ via local class field theory. We also see $\log_p$ and the $p$-adic valuation map $\ord_p$ as additive characters $K_\gP^\times\simeq \Q_p^\times \to \ob{\Q}_p$. In order to describe elements in the domain of the map (\ref{eq:localization_phi_component}) we make use of the short exact sequence of $\ob{\Q}_p[G]$-modules
	\begin{equation} \label{eq:sinnott_ses}
		\begin{tikzcd}
			0 \ar[r] & \Hom(G_K,\ob{\Q}_p) \ar[r, "A"] &  \bigoplus_{i=1}^n  \Hom(\prod_{\gP|\gp_i} K^\times_{\gP},\ob{\bQ}_p) \ar[r] & \Hom(\cO_K[\tfrac{1}{p}]^\times,\ob{\Q}_p),
		\end{tikzcd}
	\end{equation}
	where $A$ is induced by the Artin map.
	
	Let $1\leq i \leq n$ and fix a prime $\gP_i$ of $K$ above $\gp_i$. We define a basis $\{\eta_{i,\varphi},\tilde{\eta}_{i,\varphi}\}$ of the $\varphi$-component of $\Hom(\prod_{\gP|\gp_i} K^\times_{\gP},\Qp)$ as follows. First define characters $\eta_i$ and $\tilde{\eta}_i$ of $\prod_{\gP|\gp_i} K^\times_{\gP}$ by imposing that they are supported on $K_{\gP_i}^\times$ and that $\eta_{i|K_{\gP_i}^\times}=-\log_p$ and $\tilde{\eta}_{i|K_{\gP_i}^\times}=\ord_p$. We then define 
	\[\eta_{i,\varphi} = \sum_{\sigma \in G} \varphi(\sigma) \cdot \eta_i \circ \sigma, \qquad \tilde{\eta}_{i,\varphi} = \sum_{\sigma \in G} \varphi(\sigma) \cdot \tilde{\eta}_i \circ \sigma. \]
	
	Let $u_i$ be any $\gP_i$-unit of $K$ which is not a unit (take for example a generator of $\gP_i^h$, where $h$ is the class number of $K$). The choice of $u_i$ with a given $\gP_i$-valuation is unique, up to multiplication by a unit of $K$. Consider 
	$$u_{i,\varphi}= \prod_{\sigma\in G} \varphi(\sigma) \otimes \sigma^{-1}(u_i) \in \ob{\bQ} \otimes K^\times.$$
	It is clear that $u_{i,\varphi}$ is a unit away from the primes above $\gp_i$, and $u_{1,\varphi},\ldots,u_{n,\varphi}$ form a basis of $(\ob{\Q} \otimes \cO_K[1/\gp_i]^\times)[\varphi]$ modulo $(\ob{\Q} \otimes \cO_K^\times)[\varphi]$. We also fix a basis $\{\varepsilon_{1,\varphi},\ldots,\varepsilon_{r(\varphi),\varphi}\}$ of $(\ob{\Q} \otimes \cO_K^\times)[\varphi]$ modulo the kernel of Leopoldt's map $\iota_k(\varphi)$ of (\ref{eq:leopoldts_map_rho}), where $r(\varphi)=d^+(\varphi)-\delta^\textbf{L}_k(\varphi)$. For all $j=1,\ldots,n$, one can see via $\iota_{\gP_j} \colon K \hookrightarrow K_{\gP_j}=\Q_p$ the elements $u_{i,\varphi}$ and $\varepsilon_{i,\varphi}$ inside $\ob{\Q} \otimes \Q_p^\times$. We then define two matrices $L_\varphi=(L_{i,j,\varphi})$ and $M_\varphi=(M_{i,j,\varphi})$ of respective sizes $n\times n$ and $r(\varphi) \times n$ by letting 
\begin{equation}\label{eq:def_matrices_logs}
		L_{i,j,\varphi} = \dfrac{\log_p(\iota_{\gP_j}(u_{i,\varphi}))}{\ord_p(\iota_{\gP_i}(u_i))}, \qquad M_{i,j,\varphi} = {\log_p(\iota_{\gP_j}(\varepsilon_{i,\varphi}))},
\end{equation}
	where we extended $\log_p$ to $\ob{\Q} \otimes \Q_p^\times$ by linearity. Notice that $M_\varphi$ has full rank by construction. 
	
	Let $\eta'$ be an element in the $\varphi$-component of $\bigoplus_{i=1}^n  \Hom(\prod_{\gP|\gp_i} K^\times_{\gP},\ob{\bQ}_p)$, which we write as $\sum t_i\eta_{i,\varphi} + \tilde{t}_i \tilde{\eta}_{i,\varphi}$ in the basis $\{\eta_{i,\varphi}, \tilde{\eta}_{i,\varphi} \ \colon\ 1\leq i \leq n\}$. Denote by $T$ and $\tilde{T}$ the column matrices of respective coordinates $(t_1,\ldots,t_n)$ and $(\tilde{t}_1,\ldots,\tilde{t}_n)$.
	
	\begin{lemme}\label{lem:T_and_tilde_T}
		$\eta'$ belongs to the image of the map $A$ of (\ref{eq:sinnott_ses}) if and only if $\tilde{T}=L_\varphi T$ and $M_\varphi T=0$.
	\end{lemme}
	\begin{proof}
		By the exactness of (\ref{eq:sinnott_ses}) such an $\eta'$ is characterized by its vanishing at all the $u_{i}$'s and the $\varepsilon_{i,\varphi}$'s.
		The lemma then follows from a straightforward computation, using that $\eta_{j,\varphi}(\iota_{\gP_j}(u_i))=-\log_p(\iota_{\gP_j}(u_{i,\varphi}))$ and $\tilde{\eta}_{j,\varphi}(\iota_{\gP_j}(u_i))=\ord_p(\iota_{\gP_j}(u_{i}))$ for all $1\leq i,j\leq n$. 
	\end{proof}
	
	In what follows we repeatedly use the following elementary fact. For all compact topological groups $\cG$, any non-trivial continuous group homomorphism $\eta \colon \cG \to \Qp$ factors through a quotient $Z_\eta$ isomorphic to $\Zp$, and two such homomorphisms $\eta$ and $\eta'$ are proportional if and only if $Z_\eta=Z_{\eta'}$. Conversely, any topological group $Z$ isomorphic to $\Zp$ which arises as a quotient of $\cG$ defines a continuous homomorphism $\eta \colon \cG \to\Qp$, which is unique up to scaling.
	
	Fix $k_\infty \in \sZ(k)$ and put $\Gamma=\Gal(k_\infty/k)$. The above argument attaches to $\Gamma$ a non-zero element $\eta \in \Hom(G_k,\Qp)$, unique up to scaling. Since the restriction map induces an isomorphism $\Hom(G_k,\Qp) \simeq \Hom(G_K,\Qp)[\mathds{1}]$, one can write $A(\eta)$ as $\sum_i (s_i \eta_{i,\mathds{1}}+ \tilde{s}_i \tilde{\eta}_{i,\mathds{1}})$. We shall refer to the column matrices $S=(s_1,\ldots,s_n)^\textbf{t} \in \bP^{n-1}(\Qp)$ and $\tilde{S}=(\tilde{s}_1,\ldots,\tilde{s}_n)^\textbf{t}\in \bP^{n-1}(\Qp)$ as the coordinates of $k_\infty$. 
	
	\begin{proposition}\label{prop:prop_clef_non_cyc}
		\begin{enumerate}
			\item The map sending a $\Zp$-extension $k_\infty/k$ to its coordinates $S$ defines a bijection between $\sZ(k)$ and $\{S \in \bP^{n-1}(\Qp) \ \colon \ M_{\mathds{1}}S=0\}$.
			\item Let $k_\infty \in \sZ(k)$ with coordinates $S\in \ker(M_\mathds{1})$ and consider the matrix of size $(n+r(\varphi))\times n$ given in block notation by:
			\begin{equation}\label{eq:def_N_phi_S}
				\begin{bmatrix}
				N_\varphi(S)=\Diag(S)L_\varphi - \Diag(L_\mathds{1}S)\\
				M_\varphi
			\end{bmatrix},
			\end{equation}
			where $\Diag(U)$ denotes the diagonal matrix associated with a column matrix $U$.
			Then $$\delta_{k_\infty/k}^\textbf{G}(\varphi) = \dim \ker N_\varphi(S) - \gamma,$$ 
			where $\gamma=1$ if $\varphi=\mathds{1}$ and $\gamma=0$ otherwise. 
		\end{enumerate}
	\end{proposition}
	\begin{proof}
		The first point follows from Lemma \ref{lem:T_and_tilde_T} applied to $\varphi=\mathds{1}$. More precisely, $S$ and $\tilde{S}$ are uniquely determined by the relations $\tilde{S}=L_{\mathds{1}}S$ and $M_{\mathds{1}}S=0$. Conversely, any $S,\tilde{S} \in \bP^{n-1}(\Qp)$ satisfying these relations define a non-zero element $\eta \in \Hom(G_k,\Qp)$ which cuts out the Galois group of a $\Zp$-extension $k_\infty/k$.
		
		Let us prove the second claim. Fix $k_\infty \in \sZ(k)$ and denote by $\eta\in \Hom(G_K,\Qp)[\mathds{1}]$ the corresponding continuous homomorphism. By Lemma \ref{lem:expression_alternative_non-cyc_defect} and by (\ref{eq:localization_phi_component}), $\delta_{k_\infty/k}^\textbf{G}(\varphi)+\gamma$ is the dimension of the space consisting of all elements $\eta'\in \bigoplus_{i=1}^n \left( \Hom(\prod_{\gP|\gp_i} K^\times_{\gP},\ob{\bQ}_p)\right)[\varphi]$ satisfying the conditions of Lemma \ref{lem:T_and_tilde_T} and which are proportional to $A(\eta)$. This last condition means that for all $1\leq i \leq n$, the restriction to $K_{\gP_i}^\times$ of $\eta'$ and $A(\eta)$ are proportional, \textit{i.e.}, 
		$$\eta'(\varpi_{i})\cdot A(\eta)_{|\cO^\times_{K_{\gP_i}}} = A(\eta)(\varpi_{i})\cdot \eta'_{|\cO^\times_{K_{\gP_i}}},$$
		where $\varpi_{i}$ is a uniformizer of $K_{\gP_i}$.
		In terms of coordinates $S,\tilde{S}$ and $T,\tilde{T}$, this last equality is equivalent to $\tilde{t}_is_i=\tilde{s}_it_i$, an equality for all $i$ which can be rephrased as $\Diag(S)L_{\varphi}T=\Diag(L_\mathds{1}S)T$. Therefore, $\delta_{k_\infty/k}^\textbf{G}(\varphi)+\gamma$ is the dimension of the space of all $T\in \ob{\Q}_p^n$ in the kernel of both $\Diag(S)L_\varphi - \Diag(L_\mathds{1}S)$ and $M_\varphi$, as claimed.
	\end{proof}
	
	\begin{exemple}\label{ex:slope}
		If $k$ is an imaginary quadratic field, then $M_\mathds{1}$ is of size $0$ and Proposition \ref{prop:prop_clef_non_cyc} (1) provides a bijection between $\sZ(k)$ and $\bP^1(\Qp)$. The ratio ${s_1}/{s_2}\in \Qp \cup \{\infty\}$ attached to any $k_\infty\in \sZ(k)$ of coordinates $S=(s_1,s_2)\in \bP^1(\Qp)$ is referred to as the \emph{slope} of $k_\infty$. For instance, the cyclotomic extension of $k$ has slope $1$, whereas its anticyclotomic extension has slope $-1$.
	\end{exemple}
	
	\subsection{Proof of Theorem \ref{thm:non_cyclotomic_defects}}
	We keep the notations of the preceding sections and, in particular,  \eqref{eq:def_matrices_logs} and \eqref{eq:def_N_phi_S}. We abbreviate $L_\mathds{1}, M_\mathds{1}$ and $N_\mathds{1}(S)$ as $L,M$ and $N(S)$ respectively. By Proposition \ref{prop:prop_clef_non_cyc} (1), the set $\sZ(k)$ can be identified with a closed linear subvariety of $\bP^{n-1}(\Qp)$ of dimension $n-r-1=r_2+\delta^\textbf{L}_k$.  
	
	\begin{proposition}\label{prop:non_cyc_trivial_char}
		Assume that $n-r\leq 2$. \begin{enumerate}
			\item  Let $\sC(k)=\{k_\infty\in\sZ(k) \ \colon\ \delta_{k_\infty/k}^\textbf{G}\neq 0\}$. If $\sC(k)\neq \sZ(k)$, then $\sC(k)$ is finite with $|\sC(k)|\leq n-1$. If, moreover, $k$ is imaginary quadratic, then $\sC(k)=\emptyset$.
			\item Let  $\sC_\varphi(k)=\{k_\infty\in\sZ(k) \ \colon\ \delta_{k_\infty/k}^\textbf{G}(\varphi)\neq 0\}$. If $\sC_\varphi(k)\neq \sZ(k)$, then $\sC_\varphi(k)$ is finite with $|\sC(k)|\leq n$.
		\end{enumerate}
	\end{proposition}
	\begin{proof}
		Let us prove (1). Notice first that all $S\in \ker M$ lie in the kernel of $N(S)$. By Proposition \ref{prop:prop_clef_non_cyc} with $\varphi=\mathds{1}$ and $K=k$, $\sC(k)$ is in bijection with the set of all $S\in \bP^{n-1}(\Qp)$ such that $MS=0$ and $\rk N(S)<n-1$.
		
		Assume first that $k$ is quadratic, so $n=2$ and $r=0$. Given any $S=(s_1,s_2)\in \bP^1(\Qp)$, the matrix $N(S)=\Diag(S)L - \Diag(LS)$ has the form 
		\[\begin{pmatrix}
			-s_2L_{1,2} & s_1 L_{1,2} \\ s_2 L_{2,1} & -s_1 L_{2,1}
		\end{pmatrix}.\]
		As $\log_p$ is injective on $\Z_p^\times$, $\log_p(\iota_{\gp_1}(u_2))$ and $\log_p(\iota_{\gp_2}(u_1))$ both are non-zero, so at least one of the two non-diagonal entries of $N(S)$ is non-zero. Therefore, this matrix has rank one for any $S\in \bP^1(\Qp)$, and $\sC(k)=\emptyset$.
		
		We no longer assume that $k$ is imaginary quadratic, but we still assume that $n-r\leq 2$. The case where $n-r=1$ is trivial, because it forces $\sZ(k)=\{k_{\cyc}\}$. We may then assume that $n-r=2$. Since $M$ has full rank $r$, there exist invertible matrices $P,Q$ such that $PMQ=\left( I_r\ | \ 0\right)$, where $I_r$ is the identity matrix of size $r$. The change of variables $S'=Q^{-1}S$ induces a linear bijection between $\ker M \subset \bP^{n-1}(\ob{\Q}_p)$ and the projective line $\{0\} \times \bP^{1}(\ob{\Q}_p) \subset \bP^{n-1}(\ob{\Q}_p)$. Now consider the list $P_1(S'),\ldots,P_t(S')$ of all $(n-1)\times(n-1)$-minors of the matrix	
		\[N(S)=N(QS')=\begin{bmatrix}
			\Diag(QS')L- \Diag(LQS')\\
			M
		\end{bmatrix}.\]
		All the $P_k$'s are two-variable homogeneous polynomials of degree $\leq n-1$. In particular, if $\sC(k)\neq \sZ(k)$, then at least one of the $P_k$'s is not the zero polynomial and hence, it has at most $n-1$ zeros in $\{0\} \times \bP^{1}(\ob{\Q}_p)$, so we can conclude that $|\sC(k)|\leq n-1$.
		
		The proof of point (2) very similar to the previous one. Indeed, by Proposition \ref{prop:prop_clef_non_cyc}, $\sC_\varphi(k)$ is in bijection with the set of all $S\in \bP^{n-1}(\Qp)$ such that $MS=0$ and $\rk N_\varphi(S)<n$. The same argument with the $(n-1)\times(n-1)$ minors of $N(S)$ replaced by the $n\times n$ minors of $N_\varphi(S)$ shows that, if $\sC_\varphi(k)\neq \sZ(k)$, then $|\sC(k)|\leq n$.
	\end{proof}
	
	We end the proof of Theorem \ref{thm:non_cyclotomic_defects} with the case where $k$ is imaginary quadratic. We let $\tau$ be the complex conjugation of $k$. Recall that $\tau$ acts on $\varphi$ via $\varphi^\tau(g)=\varphi(\tau g \tau)$ and that $\varphi^\tau=\varphi$ if and only if $\varphi$ cuts out an extension of $k$ which is abelian over $\Q$.
	\begin{proposition}\label{prop:non_cyc_non_trivial_char}
		Assume that $k$ is imaginary quadratic and that $\varphi\neq \mathds{1}$. 
		\begin{enumerate}
			\item If $\varphi^\tau=\varphi$, then $\sC_\varphi(k)=\emptyset$.
			\item If $\varphi^\tau\neq\varphi$, then any $k_\infty\in \sC_\varphi(k)$ has transcendental slope. 
			\item If $\varphi^\tau\neq\varphi$, then $|\sC_\varphi(k)|\leq 1$. Moreover, if
			$XYZ^2-(AX-BY)(CX-DY)$
			does not vanish on any $7$-tuple in $\Lambda^7$ which form a $\ob{\bQ}$-linearly independent set, then $\sC_\varphi(k)=\emptyset$.
		\end{enumerate}
	\end{proposition}
	\begin{proof}
		Let $k_\infty\in \sZ(k)$ of coordinates $S=(s_1,s_2)$. Take $K$ to be the Galois closure over $\Q$ of the field cut out by $\varphi$. Note that $K\neq k$ and that $p$ totally splits in $K$. By Proposition \ref{prop:prop_clef_non_cyc}, $k_\infty$ belongs to $\sC_\varphi(k)$ if and only if the matrix
		\[N_\varphi(S)=\begin{pmatrix}
			s_1(L_{1,1,\varphi}-L_{1,1})-s_2 L_{1,2} & s_1L_{1,2,\varphi} \\
			s_2 L_{2,1,\varphi} & s_2(L_{2,2,\varphi}-L_{2,2}) - s_1 L_{2,1} \\
			M_{1,1,\varphi} & M_{1,2,\varphi}
		\end{pmatrix}\]
		has rank $1$. The definition of $L_\varphi$ (and also $L_{\varphi^\tau}$) involves the choice, for $i=1,2$, of a prime $\gP_i$ of $K$ above $\gp_i$ and a $\gP_i$-unit $u_i$ of non-zero valuation. Since $\tau(\gp_1)=\gp_2$, one may take  $\tau(\gP_1)=\gP_2$ and $\tau(u_1)=u_2$, so that $L_{1,1}=L_{2,2}$, $L_{1,2}=L_{2,1}$, $L_{1,1,\phi}=L_{2,2,\phi^\tau}$ and $L_{1,2,\phi}=L_{2,1,\phi^\tau}$ for $\phi\in\{\varphi,\varphi^\tau\}$. 
		We may also take $\tau(\varepsilon_{1,\varphi})$ to be $\varepsilon_{1,\varphi^\tau}$, so that $M_{1,2,\varphi}= M_{1,1,\varphi^\tau}$. Moreover, the elements $u_{1,\varphi},u_{1,\mathds{1}},u_{2,\varphi},u_{2,\mathds{1}},\varepsilon_{1,\varphi}$ and, under the additional condition $\varphi\neq \varphi^\tau$, the elements $u_{1,\varphi},u_{1,\varphi^\tau},u_{1,\mathds{1}},u_{2,\varphi},u_{2,\varphi^\tau},u_{2,\mathds{1}},\varepsilon_{1,\varphi},\varepsilon_{1,\varphi^\tau}$ form two sets of $\ob{\Q}$-linearly independent $p$-units. To see this, it suffices to consider their valuations at $\gP_1$ and $\gP_2$, and use the simple fact that units belonging to distinct isotypic components are linearly independent.
		
		Now, consider first the case where $\varphi=\varphi^\tau$. Then $N_\varphi(S)$ has rank $1$ if and only if its two columns are equal. This last condition easily implies that $s_1=\pm s_2$ and that $u_{1,\varphi}/u_{1,\mathds{1}}$, $u_{2,\mathds{1}}$ and $u_{2,\varphi}$ have $\ob{\Q}$-linearly dependent $\gP_1$-adic logarithms. But all of these units have trivial $\gP_1$-valuation, so they must be linearly dependent by Proposition \ref{prop:log_p_almost_injective}. We have already justified that this is not the case, so $N_{\varphi}(S)$ has rank $2$ and $\sC_\varphi(k)=\emptyset$. This proves (1).
		
		Assume now that $\varphi\neq \varphi^\tau$ and that $k_\infty$ has an algebraic slope, \textit{i.e.}, $s_1/s_2 \in \bP^1(\ob{\Q})$. We may assume that both $s_1$ and $s_2$ are algebraic numbers,  hence $N_\varphi(S)$ has coefficients in the $\ob{\Q}$-linear subspace $\Lambda$ of $\ob{\Q}_p$ introduced in Section \ref{sec:transcendence}. Moreover, $N_\varphi(S)$ has $\ob{\Q}$-linearly independent rows and columns. Indeed, the sets $\{\varepsilon_{1,\varphi},\varepsilon_{1,\varphi^\tau}\}$ and $\{u_{1,\varphi}/u_{1,\mathds{1}},u_{2,\mathds{1}},u_{2,\varphi},\varepsilon_{1,\varphi}\}$ are 
		two sets of independent units with trivial $\gP_1$-valuation. Therefore, their images under $\log_p \circ \iota_{\gP_1}$ are again linearly independent by Proposition \ref{prop:log_p_almost_injective}. We thus may apply Corollary \ref{coro:strong_six_exponentials} and conclude that $\rk N_\varphi(S)=2$.  This proves (2).
		
		Finally, assume that $\varphi\neq \varphi^\tau$ and that $k_\infty\in \sC_\varphi(k)$, \textit{i.e.}, $N_\varphi(S)$ has rank $1$. We already know that $s_1/s_2\in \bP^1(\Qp)-\bP^1(\ob{\Q})$ and in particular, both $s_1$ and $s_2$ are non-zero. Since $L_{1,2}\cdot M_{1,2,\varphi}\neq 0$, the vanishing of the minor obtained by removing the second row uniquely determines the slope $s_1/s_2$, so $|\sC_\varphi(k)|\leq 1$. Using the vanishing of the other minors, an elementary computation yields the polynomial relation 
		{\small
			\[M_{1,1,\varphi}\cdot M_{1,1,\varphi^\tau}\cdot L_{2,1}^2 =\left(M_{1,1,\varphi}\cdot (L_{1,1,\varphi^\tau}-L_{1,1}) - M_{1,1,\varphi^\tau}\cdot L_{2,1,\varphi}\right)  \left(M_{1,1,\varphi^\tau}\cdot (L_{1,1,\varphi}-L_{1,1}) - M_{1,1,\varphi}\cdot L_{2,1,\varphi^\tau}\right). \]
			\par}
		The elements of the set $\{u_{1,\varphi}/u_{1,\mathds{1}},u_{1,\varphi^\tau}/u_{1,\mathds{1}},u_{2,\varphi},u_{2,\varphi^\tau},u_{2,\mathds{1}},\varepsilon_{1,\varphi},\varepsilon_{1,\varphi^\tau}\}$ are linearly independent, and they all have a trivial $\gP_1$-adic valuation, so their images under $\log_p\circ \iota_{\gP_1}$ are also $\ob{\Q}$-linearly independent. Therefore, the above polynomial identity contradicts our assumption. This ends the proof of (3).
	\end{proof}

	\bibliography{bib}
	\bibliographystyle{alpha}
\end{document}